\documentclass[14pt,reqno]{amsart}
\usepackage{stix}
\usepackage{amsmath,amssymb,amsxtra,color,calligra,mathrsfs,tcolorbox}
\usepackage{xcolor}
\colorlet{mdtRed}{red!50!black}
\definecolor{dblue}{rgb}{0,0,.6}
\usepackage[colorlinks]{hyperref}
\hypersetup{linkcolor=mdtRed,citecolor=dblue,filecolor=dullmagenta,urlcolor=mdtRed}
\usepackage[all]{xy}
\usepackage{tikz,tikz-cd,enumerate}
\usetikzlibrary{matrix,arrows,decorations.pathmorphing}

\DeclareMathOperator{\Pic}{\textnormal{Pic}}

\DeclareMathOperator{\Malpha}{M^{\alpha}_X}
\DeclareMathOperator{\MHalpha}{M^{\alpha}_{X,H}}

\DeclareMathOperator{\Metaalpha}{(M^{\alpha}_X)^\eta}
\DeclareMathOperator{\gammaeta}{\gamma_{\eta}}
\DeclareMathOperator{\MHetaalpha}{(M^{\alpha}_{X,H})^{\eta}}

\DeclareMathOperator{\yeta}{Y_\eta}
\newtheorem{theorem}{Theorem}[section]
\newtheorem{lemma}[theorem]{Lemma} 
\newtheorem{proposition}[theorem]{Proposition}
\newtheorem{corollary}[theorem]{Corollary}

\theoremstyle{definition}
\newtheorem{definition}[theorem]{Definition}
\newtheorem{remark}[theorem]{Remark}

\numberwithin{equation}{section} 

\begin{document}
	
	\baselineskip=15.5pt 
	
	\title[Chen--Ruan 
	cohomology and orbifold Euler characteristic of moduli spaces of parabolic bundles]{Chen--Ruan 
		cohomology and orbifold Euler characteristic of moduli spaces of parabolic bundles}
	
	\author[I. Biswas]{Indranil Biswas}
	
	\address{Department of Mathematics, Shiv Nadar University, NH91, Tehsil
		Dadri, Greater Noida, Uttar Pradesh 201314, India}
	
	\email{indranil.biswas@snu.edu.in, indranil29@gmail.com}
	
	\author[S. Chakraborty]{Sujoy Chakraborty}
	
	\address{Department of Mathematics, Indian Institute of Technology Chennai, Chennai, India}
	\email{sujoy.cmi@gmail.com}
	
	\author[A. Dey]{Arijit Dey}

	\address{Department of Mathematics, Indian Institute of Technology Chennai, Chennai, India}
	\email{arijitdey@gmail.com}

	\begin{abstract}
		We consider the moduli space of stable parabolic Higgs bundles of rank $r$ and fixed determinant, and having 
		full flag quasi-parabolic structures over an arbitrary parabolic divisor on a smooth connected complex 
		projective curve $X$ of genus $g$, with $g\,\geq\, 2$. The group $\Gamma$ of $r$-torsion points
		of the Jacobian of $X$ acts on this moduli space. We describe the connected components of the various fixed point loci of this moduli under non-trivial elements from $\Gamma$.
		When the Higgs field is zero, or in other words when we restrict ourselves to the moduli of stable parabolic bundles, we also compute the orbifold Euler characteristic of the corresponding global quotient orbifold.
		We also describe the Chen--Ruan cohomology groups of this orbifold under certain conditions on the rank and degree, and describe the Chen--Ruan product structure in special cases.
	\end{abstract}
	
	\maketitle
	
	\tableofcontents
	
	\section{Introduction}
	
	Orbifolds are one of the simplest kind of singular spaces. They arise naturally in many branches of mathematics.
	One such instance is symplectic reductions, where the resulting space is often
	an orbifold. They also appear naturally in string theory, where many known Calabi-Yau threefolds appear as
	crepant resolution of some Calabi-Yau orbifold. In string theory, the Euler characteristic of orbifolds plays an important
	role. In \cite{DHVW85} its authors concluded that for quotient orbifolds of the form $M/G$, where $G$ is
	a finite group acting on a compact manifold $M$, the correct Euler characteristic
	in the context of string theory is the so-called \textit{orbifold Euler characteristic}:
	\begin{align}\label{orbifold-euler-char-definition}
		\chi_{orb}(M,\,G) \,\,:=\,\, \frac{1}{|G|}\sum_{gh=hg}\chi\left(M^{\langle g,h\rangle}\right)
		\,\,=\,\, \sum_{[g]} \chi\left(M^g/C(g)\right)\,.
	\end{align}
	Here the second sum is taken over a set of representatives for the conjugacy classes of $G$; the
	above equality follows from \cite{HH90}. In \eqref{orbifold-euler-char-definition} $M^{\langle g,\,h\rangle}$ denotes the common fixed points for $g$ and $h$,
	while $M^g$ denotes the fixed point set (which is a submanifold) and $C(g)$ denotes the centralizer of $g$
	with $\chi(M^g/C(g))$ denoting the usual topological Euler characteristic of the quotient space.
	It is worth pointing out that in some situations when $M/G$ has a resolution of singularities
	$$\widetilde{M/G}\,\longrightarrow\, M/G$$
	such that the canonical bundle of $\widetilde{M/G}$ is trivial, the above defined
	$\chi_{orb}(M,\,G)$ coincides with the usual Euler characteristic of the resolution $\widetilde{M/G}$.
	
	On the other hand, the Chen--Ruan cohomology ring of an orbifold was introduced in \cite{CR04} as the degree zero part of the small quantum cohomology ring of the orbifold constructed by the same authors \cite{CR02}. It contains the usual cohomology ring of the orbifold as a subring. In general, the Chen--Ruan cohomology group has a \textit{rational} grading, where the grading is shifted by the \textit{degree-shift numbers} (\S~\ref{chen-ruan cohomology section}) (also known as Fermionic shifts; these numbers also appear in the study of mirror symmetry related problems for Higgs bundles). When the orbifold has an algebraic structure, then the Betti numbers of the Chen--Ruan cohomology are invariant under certain crepant resolutions \cite{Y04}, and this makes the study of these groups important in Calabi-Yau geometry. It is conjectured by Ruan \cite{R02} that the Chen--Ruan cohomology ring is isomorphic to the cohomology ring of a smooth crepant resolution if both the orbifold and the resolution are hyper-K\"ahler. In case of a global quotient orbifold $M/G$ for a finite group $G$, its Chen--Ruan cohomology groups are described in terms of the cohomologies of the connected components of the fixed point loci 
	$M^g$ for various $g\in G$ (cf. \cite{FG03}). The degree-shift numbers may vary with these components.\\
	It is of particular interest to study the Chen--Ruan cohomology and the orbifold Euler characteristic for orbifolds arising in the context of moduli spaces. To this end, let $X$ be an irreducible smooth complex projective curve of genus $g\,\geq \,2$. Fix an integer $r\geq 2$ and a line bundle $\xi$ of degree $d$ on $X$. If $g=2,$ we assume that $r\,\geq\, 3$. let $M(r,d)$ denote the moduli space of stable vector bundles on $X$ of rank $r$ and degree $d$; it is a smooth quasi-projective variety. Let $M_{\xi}(r,d)\, \subset\, M(r,d)$ denote the moduli space of stable vector bundles on $X$ with rank $r$ and determinant $\xi$. The Jacobian of $X$, denoted by $\Pic^0(X)$, is an Abelian variety of dimension $g$. Its $r$-torsion points form a finite group isomorphic to $(\mathbb{Z}/r\mathbb{Z})^{2g}$. Let us denote this group by $\Gamma$. The group $\Gamma$ acts on $M(r,d)$ by tensorization with line bundles, and this action
	preserves each $M_\xi(r,d)$. The Chen--Ruan cohomology of the quotient orbifolds $M_\xi(r,d)/\Gamma$ has been described in \cite{BP08,BP10} when $r$ is prime.
	
	In this paper, we consider the moduli space of parabolic stable vector bundles, which were introduced and constructed by Mehta and Seshadri in \cite{MS80}. These are vector bundles $E$ on $X$ with a filtration of fibers of $E$ at a collection of finitely many points $S$ of $X$ and certain real numbers, called weights, 
	attached to these filtrations (see $\S~2$ for details). 
	Let $M^\alpha_{\xi}(r,d)$ denote the moduli space of rank $r$ stable parabolic
	vector bundles $E_*$ on $X$ whose underlying bundle $E$ has determinant $\xi\,\in\,
	\Pic^d(X)$ and the parabolic structure at each parabolic point $p_i\,\in\, S$ is of the full-flag type. We also assume
	that the weights are \textit{generic}, so that parabolic semistability is equivalent to parabolic stability, which makes $M^\alpha_{\xi}(r,d)$ a compact complex manifold. The
	group $\Gamma$ acts on $M^{\alpha}_\xi(r,d)$ by tensorization as before. While the introduction of
	parabolic structures bring interesting new phenomenon, it will also be apparent that working
	with parabolic moduli has some technical advantage over the usual vector bundles (i.e., those
	without parabolic structures; see Lemma \ref{parabolic fixed point}).
	
	Recently there has been some progress in understanding the Chen--Ruan cohomology groups for the orbifolds $M^{\alpha}_\xi(r,d)/\Gamma$ 
	in some special cases. In \cite{BD10}, these groups were studied under the assumptions that the rank $r=2$, degree $d=1$ and the 
	weights are \textit{sufficiently small} (in the sense of \cite[\S~5]{BY99}). The case when the rank is any prime number (with $r,d$ 
	coprime and the weights are small) was studied in \cite{BDS22}, but their description lacks a satisfactory determination of the 
	degree-shift numbers (see [Remark 4.5, \textit{loc. cit.}]).
	
	The study of these quantities requires an understanding of the connected components of their fixed point loci under the action of elements in $\Gamma$. 
	In fact, our discussion up to Section \ref{connected componets of fixed point loci} is valid for the more general case of the moduli of parabolic stable Higgs bundles, which are pairs $(E_*,\varphi_*)$ with $E_*$ a parabolic bundle and $\varphi_* :  E_*\rightarrow E_*\otimes K_X(D)$ is a parabolic morphism called Higgs field, satisfying certain stability conditions (see \S\ref{se2} for details). Here $D$ is the divisor given by the sum of points in $S$. Let $M^{\alpha}_{\xi,H}(r,d)$ denote the moduli of such pairs. Our first main result is to provide a description of the connected components of the fixed-point loci $M^{\alpha}_{\xi,H}(r,d)^\eta$ for non-trivial $\eta\in\Gamma$. This is done in Section \ref{connected componets of fixed point loci}, which takes input from Section \ref{fixed point loci section} (more specifically, Lemma \ref{parabolic fixed point}) where we describe these fixed point loci as a finite group quotient of certain varieties, which are in some sense a generalization of Prym varieties for spectral covers associated to these $\eta$. 
	Our next main result is the computation of the orbifold Euler characteristic of the quotient orbifold $M^{\alpha}_{\xi}(r,d)/\Gamma$. We achieve this by computing the cohomology groups of the quotients $M^{\alpha}_{\xi}(r,d)^{\eta}/\Gamma$ in terms of the cohomology of Prym and cohomology of parabolic moduli for smaller ranks (Proposition \ref{equivariant cohomology proposition}). As a consequence, we obtain the following:
	
	\begin{theorem}[\textnormal{Theorem \ref{orbifold euler characteristic corollary}}]
		The orbifold Euler characteristic (\ref{orbifold euler characteristic}) of $M^{\alpha}_{\xi}(r,d)/\Gamma$ is given by
		$$\chi_{orb}\left(M^{\alpha}_{\xi}(r,d)\,,\,\Gamma\right) = \chi\left(M^{\alpha}_{\xi}(r,d)\right)\,,$$
		where the right-hand side denotes the usual Euler characteristic.
	\end{theorem}
	The reason for taking the Higgs field to be zero from Section \ref{orbifold euler characteristic section} onwards is that the $\Gamma$--action on the cohomology of the parabolic Higgs moduli is not well understood yet; for example, the action on the cohomology of $M^{\alpha}_{\xi,H}(r,d)$ is expected to be non-trivial (which is the case for usual Higgs moduli, cf. \cite{HP12}), whereas a key ingredient in the above proof is that $\Gamma$--action on the cohomology of the parabolic moduli $M^{\alpha}_{\xi}(r,d)$ is trivial \cite[Proposition 4.1]{BD10}.\\
	
	Finally, in Section \ref{chen-ruan cohomology section}, we study the Chen--Ruan cohomology groups for the orbifold $M^{\alpha}_\xi(r,d)/\Gamma$ 
	when the rank $r$ is a product of \textit{distinct} primes, $r$ and $d$ are coprime, and for \textit{any} system of generic weights. As discussed above, apart from understanding the connected components of the fixed point loci $M^{\alpha}_\xi(r,d)^{\eta}$ for $\eta\in\Gamma$, one of the other main challenges is the computation of their degree-shift numbers. This is done in Corollary \ref{degree-shift number collary}, where the conditions on $r$ and $d$ mentioned in the beginning of the paragraph plays a crucial role (cf. Proposition \ref{Gamma action on components}). We conclude Section \ref{chen-ruan cohomology section} with a study of the Chen--Ruan product. The multiplicative structure is obtained using a certain Chen--Ruan poincar\'e pairing, which we describe in our context (Definition \ref{chen-ruan pairing}). We conclude the section by giving a necessary condition for the product to be nonzero in certain cases (see Lemma \ref{chen-ruan product lemma} and the corollary following it).\\
	It is worth mentioning that the degree-shift numbers in the more general case of parabolic Higgs moduli is already known; in Higgs 
	context these numbers are also referred to as Fermionic shifts (see \cite[Page 16]{GO18}). The determination of these numbers for parabolic 
	Higgs case is much easier due to the presence of a particular type of symplectic structure on those moduli, which is unfortunately not 
	available in our case, which makes their determination more difficult and interesting.
	
	\section{Preliminaries}\label{se2}
	
	Let $X$ be an irreducible smooth complex projective curve (equivalently, a compact connected Riemann surface) of genus
	$g$, with $g\,\geq\, 2$. Vector bundles on $X$ will always mean to be algebraic vector bundles.
	Fix an integer $r\,\geq\, 2;$ if $g\,=\,2,$
	then set $r\,\geq\, 3$. Fix a line bundle $\xi$ on $X$ of degree $d$. By a point of $X$ we will always mean a closed point. Fix a
	finite subset $$S\,=\,\{p_1,\,p_2,\,\cdots,\,p_s\}\, \subset\, X$$ of $s$ distinct points; they will be referred to
	as \textit{parabolic points}.
	Let $E$ be a vector bundle of rank $r$ on $X$. The fiber of $E$ over any point $z\, \in\, X$ will be denoted by $E_z$. 
	
	\begin{definition}
		A \textit{parabolic data of rank r} for points of $S$ consists of the following collection: for each $p_i\,\in \,S$
		\begin{itemize}
			\item a string of positive integers $(m^i_1,\,m^i_2,\,\cdots,\,m^i_{l_i})$ such that $\sum_{j=1}^{l_i}m^i_j \,=\, r$, and
			
			\item an increasing sequence of real numbers $0\,\leq\, \alpha^i_1\,<\,\alpha^i_2\,<\,\cdots\,<\,\alpha^i_{l_i}\,<\,1$.
		\end{itemize}
		
		If $m^i_j \,=\, 1$ for all $1\,\leq\, j\,\leq \,l_i$ and $1\,\leq \,i\,\leq \,s$, we say that it is a
		\textit{full-flag} parabolic data; in that case $l_i\,=\, r$ for all $i$.
		
		A \textit{parabolic structure} on $E$ over $S$, with parabolic data of rank $r$ as above, consists of the following:
		for each $p_i\,\in\, S$, a weighted filtration of subspaces
		\begin{align*}
			E_{p_i} \,=\, E^i_1\,\supsetneq\, E^i_2\,\supsetneq\, &\,\cdots\, \supsetneq\, E^i_{l_i}\,\supsetneq\, E^i_{l_i+1} \,=\, 0\\
			0\,\leq\, \alpha^i_1\,<\,\alpha^i_2\,<\,&\,\cdots\,<\,\alpha^i_{l_i}\,<\,\alpha^i_{l_i+1} \,=\, 1
		\end{align*}
		such that $m^i_j \,=\, \dim(E^i_j/E^i_{j+1})$ for all $1\,\leq\, j\,\leq\, l_i$ and all $1\,\leq \,i\,\leq\, s$.
		
		The above collection 
		$\alpha\,:=\,\{(\alpha^i_1\,<\,\alpha^i_2\,<\,\cdots\,<\,\alpha^i_{l_i})_{ 1\leq i\leq 
			s}\}$ is called the \textit{weights}, and the above integer $m^i_j$ is called the 
		\textit{multiplicity} of the weight $\alpha^i_j$.
		
		By a \textit{parabolic bundle} we mean a collection of data $(E,\, m,\, \alpha)$, where $E$ is a vector bundle on $X$, while $m$ and $\alpha$ are as 
		described above. For notational convenience, such a parabolic vector
		bundle will also be referred to as $E_*$; the vector bundle
		$E$ is called the underlying bundle for $E_*$. A full-flag parabolic data is also called a \textit{full-flag} parabolic structure.
	\end{definition}
	
	\begin{definition}
		Let $E_*$ be a parabolic bundle on $X$ of rank $r$. The \textit{parabolic degree} of $E_*$ is defined as 
		\[\textnormal{pardeg}(E_*)\,:=\, \textnormal{deg}(E)+\sum_{i=1}^{s}\sum_{j=1}^{l_i}m^i_j\alpha^i_j\, \in\, {\mathbb R}\] 
		and the \textit{parabolic slope} of $E_*$ is defined as 
		\[\textnormal{par}\mu(E_*)\,:=\, \dfrac{\textnormal{pardeg}(E_*)}{r}\, \in\, {\mathbb R}.\]
		The parabolic bundle $E_*$ is called \textit{parabolic semistable} (respectively,
		\textit{parabolic stable}) if for every proper sub-bundle $F\, \subsetneq\, E$ we have 
		\[
		\textnormal{par}\mu(F_*)\,\leq\, \textnormal{par}\mu(E_*)\ \ \, \left(\text{respectively, }\
		\textnormal{par}\mu(F_*)\, <\, \textnormal{par}\mu(E_*)\right),\]
		where $F_*$ denotes the parabolic bundle defined by $F$ equipped with the parabolic structure
		induced by the parabolic structure of $E_*$ (see \cite{MS80} for
		the details).
	\end{definition}
	
	\begin{definition}\label{def:generic-weight}
		A system of weights $\alpha$ is called \textit{generic} if every parabolic semistable bundle $E_*$, of given rank and
		degree, with weights $\alpha$ is parabolic stable. We refer to \cite{BY99} for more details.
	\end{definition}
	
	\begin{remark}\label{parabolic direct sum}
		See \cite{Ses}, \cite{MS80} for homomorphisms between parabolic bundles. The class of parabolic semistable bundles with
		fixed parabolic slope forms an abelian category. We refer to \cite[p.~ 68]{Ses} for further details.
	\end{remark}
	
	\begin{definition}
		Let $E_*$ and $E'_*$ be two parabolic vector bundles with parabolic data $(m,\,\alpha)$ and $(m',\,\alpha')$
		respectively. Let $\varphi \,:\, E\,\longrightarrow\, E'$ be a morphism of vector bundles. $\varphi$ is called a \textit{parabolic morphism} if, for each parabolic point $p$,
		$$\alpha_i^p\,>\,{\alpha'}_j^p \,\,\implies\,\, \varphi_p(E_i^p)\,\subset\, {E'}^p_{j+1}\,\,.$$
		We shall mean $\varphi_*$ to denote a parabolic morphism, and $\varphi$ to denote the corresponding morphism on the underlying bundles.
	\end{definition}
	
	\begin{definition}
		(i) A \textit{ parabolic} Higgs bundle on $X$ is a parabolic vector bundle $E_*$ together with a parabolic morphism
		\[\varphi_*\,\,:\,\,E_*\,\longrightarrow\, E_*\otimes K_X(D)\]
		which is called \textit{Higgs field}. This will be usually denoted by $(E_*\,,\,\varphi_*)$.\\			
		\hspace{0.85in}(ii) A (twisted) parabolic ${\rm SL}(r,k)$--Higgs bundle of fixed determinant $\xi$ on $X$ is
		a parabolic Higgs bundle $(E_*,\,\varphi_*)$ such that $\det(E)\,\simeq\, \xi$ and $\text{tr}(\varphi)\,=\,0$.\\
		The \textit{parabolic degree} and \textit{parabolic slope} of $(E_*,\varphi_*)$
		are defined to be the corresponding numbers for the underlying parabolic bundle $E_*$.
	\end{definition}
	
	\subsection{Parabolic push-forward and pull-back}\label{parabolic pushforward}
	
	Let $X$ and $Y$ be two irreducible smooth projective curves, and let $\gamma\,:\, 
	Y\,\longrightarrow\, X$ be a finite \'etale Galois morphism. If $F$ is a vector bundle on 
	$Y$ of rank $n$, then $\gamma_*F$ is a vector bundle on $X$ of rank $mn$, where $m$ is the 
	degree of the map $\gamma$. Given a parabolic structure on $F$, there is a natural way to 
	construct a parabolic structure on $\gamma_*F$. We refer to \cite[\S~3]{BM19} for details.
	
	Let us mention a special case of parabolic push-forward which will be used here. For 
	simplicity, first assume that there is only one parabolic point on $X$, i.e., $S\,=\,\{p\}$. 
	Let $m$ be the degree of $\gamma$. Since $\gamma$ is unramified, the inverse image 
	$\gamma^{-1}(p)$ consists of $m$ distinct points of $Y$. Suppose we are given a full-flag 
	parabolic data on $Y$ of rank $n$ with $\gamma^{-1}(p)$ as the set of parabolic points, so that 
	the parabolic structure on a vector bundle $F$ of rank $n$ on $Y$ is of the form 
	\begin{align*}
		q\,\in\, \gamma^{-1}(p),\,\,F_q \,=\, F^q_1\,&\supsetneq\, F^q_2\,\supsetneq \,\cdots\,\supsetneq\, F^q_n
		\,\supsetneq\, F^q_{n+1}\,=\,0,\\
		\alpha^q_1\,&<\,\alpha^q_2\,<\,\cdots\cdots<\,\alpha^q_n\,<\,\alpha^q_{n+1}\,=\,1.
	\end{align*}
	Moreover, we assume that in the entire collection $\{\alpha^q_j\,\mid\,\,q\,\in\,\gamma^{-1}(p),\, \,1\,\leq\, j\,\leq\, n\}$
	all numbers are distinct.
	We shall construct, from this data, a full-flag parabolic structure on $E\,=\,\gamma_*F$ at the point $p$. Note that
	\[E_p\,=\,\bigoplus_{q\,\in\, \gamma^{-1}(p)}F_{q}\,.\]
	Let $$\beta^p_1\,<\,\beta^p_2\,<\,\cdots\,<\,\beta^p_{mn}$$ be the increasing sequence of length $mn$
	obtained by ordering the numbers $\{\alpha^q_j\,\mid\,\,q\,\in\,\gamma^{-1}(p),\,1\,\leq\, j\,\leq\, n\}$. For each integer
	$1\,\leq\, k\,\leq\, mn$, define the subspace
	$$E^p_k \,:=\, \bigoplus_{q\in\gamma^{-1}(p)}F^q_{\omega(q,k)}\, \subset\, E_p\,,$$ 
	where $\omega(q,k)$, for each point $q$, is the smallest integer $1\,\leq\, j(q)\,\leq\, n$ satisfying the
	condition $\beta^p_k\,\leq\, \alpha^q_{j(q)}$. So we have a filtration of $E_p$ by the subspaces $\{E^p_k\}_{k=1}^{mn}$.
	It is straight-forward to see that $\dim(E^p_k/E^p_{k+1}) \,=\, 1$ for all $1\,\leq\,
	k\, \leq\, mn-1$ and $\dim E^p_{mn}\,=\, 1$. Consequently,
	\begin{align*}
		E_p \,=\, E^p_1&\,\supsetneq\, E^p_2\,\supsetneq\,\cdots\,\supsetneq\, E^p_{mn}\\
		\beta^p_1&\,<\,\beta^p_2\,<\,\cdots\cdots\,<\,\beta^p_{mn}
	\end{align*}
	is a full-flag parabolic structure of rank $mn$ at $p$. Finally, for multiple parabolic points on $X$, we perform exactly
	the same construction individually for each parabolic point to define the parabolic push-forward.
	
	Let $\gamma\,:\,Y\,\longrightarrow\, X$ be as above. If $E_*$ is a parabolic bundle on $X$, then $\gamma^*E$ has an induced
	parabolic structure as follows: let $S\,\subset\, X$ be the parabolic points for $E$; we define $\gamma^{-1}(S)$ as the set
	of parabolic points for $\gamma^*E$, and for any $p\in S$ and $q\in \gamma^{-1}(p)$, give the fiber
	$(\gamma^*E)_q\,=\,E_p$ the same parabolic structure as $E_p$. 
	
	The direct image and pull-back of parabolic Higgs bundles is also well-defined in this setting. Since $\gamma$ is \'etale,
	we have $K_Y\,\simeq\, \gamma^* K_X$. Thus if $\varphi\,:\, E_*\,\longrightarrow \,E_*\otimes K_X(D)$ is a Higgs field on $E_*$
	on $X$, then $$\gamma^*(\varphi_*)\,:\, \gamma^*(E_*)\,\longrightarrow \,\gamma^*(E_*)\otimes K_Y(\gamma^*D)$$ is a
	Higgs field on the pull-back $\gamma^*(E_*)$. Clearly, if $\varphi_*$ is  parabolic, then $\gamma^*(\varphi_*)$ is  parabolic as well.
	
	Similarly, if $\phi_*\,:\, F_*\,\longrightarrow\, F_*\otimes K_Y(\gamma^*D)$ is a Higgs field for $F_*$, then apply
	push-forward followed by projection formula:
	$$\gamma_*(\phi_*)\,:\,\gamma_*(F_*)\,\longrightarrow\, \gamma_*(F_*\otimes K_Y(\gamma^*(D))\,\simeq\, \gamma_* (F_*)\otimes K_X(D)$$
	and consider $\gamma_*(\phi_*)$ as a Higgs field on $\gamma_*(F_*)$, where $\gamma_*(F_*)$ has the induced parabolic structure from
	$F_*$ as described above. For a parabolic point $p$, if $\gamma^{-1}(p) \,=\, \{q_1,\,q_2,\,\cdots,\,q_m\}$, then 
	$$(\gamma_*\phi)_p \,=\, \bigoplus_{i=1}^{m} \phi_{q_i}\,\,.$$
	It follows easily from this description that if the Higgs field $\phi_*$ is parabolic then $\gamma_*(\phi_*)$ is parabolic as well.
	
	\section{The fixed point loci}\label{fixed point loci section}
	
	We fix a rank $r$, a subset of parabolic points of $X$
	and a line bundle $\xi$ of degree $d$ on $X$. 
We adopt the following notation:
	\begin{align*}
		\MHalpha \,:=\, &\,\,\text{Moduli space of stable  parabolic}\,{\rm SL}(r,{\mathbb C})-\text{Higgs bundles of rank}\,\,r,\,\,\text{fixed}\\ &\text{determinant}\,\,\xi\,\,\text{and weights}\,\,\alpha\,\,\text{of full-flag type on}\,\,X.\\
		\Gamma \,:=\, &\,\,\{L\,\in \,\text{Pic}^0(X)\,\,\mid\,\,L^r\,\simeq\, \mathcal{O}_X\}.
	\end{align*}
	We aim to study the connected components of the fixed point loci of the automorphism of $\MHalpha$
	defined by tensoring with any $r$-torsion line bundle $\eta$. The following result in linear algebra which will be used.
	
	\begin{lemma}\label{linear algebra result}
		Let $\psi\, \in\, {\rm GL}(V)$ be a diagonalizable automorphism of a vector space $V$ of dimension $r$ equipped with a filtration of subspaces
		\[V\,=\,V_r\,\supsetneq\, V_{r-1}\,\supsetneq\, V_{r-2}\,\supsetneq\, \cdots\,
		\supsetneq\, V_1\,\supsetneq\, 0\]
		such that $\psi(V_i)\,=\, V_i\,\,\, \forall\,\,\, 1\, \leq\, i\, \leq\, r$.
		Then there exists a basis of of $V$ consisting of eigenvectors
		$\{v_1,\,v_2,\,\cdots,\,v_r\}$ of $\psi$ such that
		$V_j\,=\,\langle v_1,\,\cdots,\, v_j\rangle\,\,\, \forall \,\,\, 1\, \leq\, j\, \leq\, r$.
	\end{lemma}
	
	\begin{proof}
		Since $\psi$ is diagonalizable, for any subspace $W\, \subset\, V$ such that
		$\psi(W)\,=\, W$, the restriction of $\psi$ to $W$ is diagonalizable and, moreover, since diagonalizable maps are semisimple, there
		is a subspace $W'\, \subset\, V$ satisfying the conditions that $\psi(W')\,=\, W'$ and
		$V\,=\, W\bigoplus W'$. Choose any basis vector $v_1$ for $V_1$. Suppose $\{v_1,\,v_2,\,\cdots,\,v_j\}$ has been chosen
		satisfying the hypothesis up to some $j$ with $j\,\leq\, r-1$. Then there exists a vector $v_{j+1}\,\in\, V_{j+1}$ such that
		$V_{j+1}\,=\,\langle v_{j+1}\rangle \bigoplus V_j$. Now $\{v_1,\,\cdots,\,v_{j+1}\}$ satisfy the hypothesis till $j+1$. Repeating this process, the lemma follows.
	\end{proof}
	
	The group $\Gamma$ acts on $\MHalpha$. The action of $\eta\,\in\,\Gamma$ sends a parabolic Higgs bundle
	$(E_*,\,\varphi_*)$, where $\varphi_*\,:\,E_*\,\longrightarrow\, E_*\otimes K_X(D)$ is a Higgs field, to the
	parabolic Higgs bundle $$(E_*\otimes\eta,\,\varphi_*\otimes\text{Id}_{\eta})\,.$$
	A parabolic Higgs bundle $(E_*,\,\varphi_*)$ is a \textit{fixed point} under this action of $\eta$ if there exists an
	isomorphism of parabolic Higgs bundles between $(E_*,\,\varphi_*)$ and $(E_*\otimes\eta,\,\varphi_*\otimes\text{Id}_{\eta})$, namely
	a parabolic isomorphism
	$$\psi_* \,:\, E_*\,\simeq\, E_*\otimes\eta$$
	which is compatible with the Higgs fields in the sense that the following diagram commutes:
	\begin{align*}
		\xymatrix@=1.8cm{ E_*\ar[r]^{\psi_*} \ar[d]_{\varphi_*} & E_*\otimes \eta \ar[d]^{\varphi_*\otimes{\rm Id}_\eta} \\
			E_*\otimes K_X(D) \ar[r]^(0.45){\psi_*\otimes{\rm Id}_{K_X(D)}} & E_*\otimes K_X(D)\otimes \eta
		}
	\end{align*}
	
	For $\eta\,\in \,\Gamma\setminus\mathcal{O}_X$, let 
	$\MHetaalpha\,\subset\,\MHalpha$ be the locus of fixed points for the action of $\eta$ on $\MHalpha$. 
	If $m\,=\,\text{ord}(\eta)$, choosing a nonzero section $s_0\, \in\, H^0(X,\, \eta^{\otimes m})$,
	define the spectral curve
	\begin{equation}\label{ye}
		Y_\eta\, :=\, \{v\, \in\, \eta\,\, \mid\,\, v^{\otimes m}\, \in\, s_0(X)\}\, .
	\end{equation}
	The natural projection $\gammaeta\,:\,Y_\eta\,\longrightarrow\, X$ is an \'etale Galois covering
	with Galois group ${\mathbb Z}/m{\mathbb Z}$. The isomorphism class of this
	covering $\gammaeta$ does not depend on the choice of the section $s_0$.
	
	\subsection{Description of fixed point loci as quotients}\label{estimate}
	
	Given a set $\alpha\,=\,\{\alpha_1,\,\alpha_2,\,\cdots,\, \alpha_r\}$ of $r$ distinct 
	elements, and a divisor $m$ of $r$, let $ {\bf P}(\alpha)$ denote the set of all possible partitions of its elements into 
	$m$ subsets, each containing $l\,=\,r/m$ elements. Clearly, we have $| {\bf P}(\alpha)|\,=\,
	{r\choose l}{r-l \choose l}{r-2l\choose l}\cdots {l\choose l} = \dfrac{r!}{(l!)^{m}}.$
	
	Consider the spectral curve $\gammaeta\,:\, Y_\eta\,\longrightarrow \,X$ in \eqref{ye}. Given a full-flag 
	parabolic data of rank $r$ at the parabolic points $S$ of $X$, and an element of ${\bf P}(\alpha)$,
	we would like to describe a full-flag parabolic data of rank $l$ on $\gammaeta^{-1}(S)$.
	
	For simplicity, let us start with the case of a single parabolic point $S\,=\,\{p\}$. So, 
	we are given a full-flag parabolic data of rank $r$ at $p$. Let $\alpha$ denote its set of 
	weights. Denote $\mu\,=\, \exp(2\pi\sqrt{-1}/m)$, and we have
	$$\textnormal{Gal}(\gammaeta)\,=\,\{1, 
	\,\mu,\,\mu^2,\,\cdots,\,\mu^{m-1}\}\,\subset \,{\mathbb C}^*\,=\, {\mathbb C}\setminus\{0\}.$$ The action of of 
	$\textnormal{Gal}(\gammaeta)$ on $\gammaeta^{-1}(p)$ is via multiplication of ${\mathbb C}^*$ on
	$\eta$. Using this, we can 
	define a full-flag parabolic data at the points of $\gammaeta^{-1}(p)$ as follows: fix 
	an ordering on the points of $\gammaeta^{-1}(p)$, say 
	$\gammaeta^{-1}(p)\,=\,\{q_1,\,q_2,\,\cdots,\,q_m\}$, such that $\mu^i$ acts on 
	$\gammaeta^{-1}(p)$ as the cyclic permutation sending $q_j$ to $q_{j+i}$\,, where the 
	subscript $(j+i)$ is to be understood mod $m$. For $\textbf{t}\,\in\,  {\bf P}(\alpha)$, suppose \[\alpha 
	\,=\, \coprod_{j=1}^{m} \Lambda_j \] be the partition of the set of weights $\alpha$ 
	according to $\textbf{t}$. Clearly each $\Lambda_j$ can be arranged into an increasing sequence of 
	length $l$. We designate $\Lambda_j$ as the set of weights at $q_j$ for each $1\,\leq\, 
	j\,\leq\, m$. This gives a full-flag parabolic data of rank $l$ at the points of 
	$\gammaeta^{-1}(p)$. Finally, for multiple parabolic points, say $S\,= 
	\,\{p_1,\,\cdots,\,p_s\}$ and $\gammaeta^{-1}(p_i)
	\,=\,\{q_{1i},\,q_{2i},\,\cdots,\,q_{mi}\}$, 
	perform the same procedure as above for each $p_i$ to get the parabolic structure 
	on $\gammaeta^{-1}(p_i)$. In particular, the number of possible parabolic data on $\gammaeta^{-1}(S)$ is
	$| {\bf P}(\alpha)|\, = \, \left({r\choose l}{r-l \choose l}{r-2l\choose l}\cdots {l\choose l}\right)^s = \left(\dfrac{r!}{(l!)^{m}}\right)^s.$
	
	For each $\textbf{t}\,\in\,  {\bf P}(\alpha)$, let $M^\textbf{t}_{Y_\eta,H}(l,d)$ denote the moduli space of
	stable  parabolic ${\rm GL}(l,{\mathbb C})$-Higgs bundles over $Y_\eta$ of  degree $d$ and having full-flag
	parabolic structures at the points of $\gammaeta^{-1}(p)$ according to $\textbf{t}$ as described above. Let
	$\mathcal{N}^\textbf{t}_\eta\,\subset\, M_{Y_\eta,H}^{\textbf{t}}(n,d)$ denote the subvariety consisting of all those Higgs
	bundles $(F_*,\phi_*)$ over $Y_\eta$ such that $\det (\gammaeta_*F)\,\simeq\, \xi$ and $tr(\gammaeta_*(\phi))\,=\,
	0$. Define $$\mathcal{N}_\eta\,=\,\underset{\textbf{t}\in  {\bf P}(\alpha)}{\coprod}\mathcal{N}^\textbf{t}_\eta.$$
	
	\begin{lemma}\label{parabolic fixed point}
		There is a surjective morphism $f\,:\,\mathcal{N}_\eta\,\longrightarrow\, \MHetaalpha$ given by parabolic
		pushforward by the \'etale map $\gammaeta$.
	\end{lemma}
	
	\begin{proof}
		First, assume for simplicity that there is only one parabolic point $S\,=\,\{p\}$.
		
		The map $f$ sends any $(F_*,\, \phi_*)$ to the parabolic pushforward $\gammaeta_*((F_*,\, \phi_*))$.
		We claim that $f((F_*,\, \phi_*))$ is a parabolic semistable Higgs bundle.
		
		To prove this, if $(E_*,\,\varphi_*)\,= \,f((F_*,\,\phi_*))$, we have 
		\begin{align}\label{eqn:pullback-isomorphism}
			\gammaeta^*((E_*,\,\varphi_*))
		\,\cong\, \bigoplus_{\sigma\in \textnormal{Gal}(\gammaeta)} \sigma^*((F_*,\,\phi_*))\ ,
		\end{align} 
		where $\sigma^*F_*$ has the
		obvious parabolic Higgs bundle structure coming from $(F_*,\,\phi_*)$, induced by pulling back via $\sigma$. It is also clear
		that $\text{par}\mu(\sigma^*F_*)\,=\,\text{par}\mu(F_*)$ for all $\sigma$. Consequently, $\gammaeta^*((E_*,\,\varphi_*))$ is a
		direct sum of stable  parabolic Higgs bundles of same parabolic slope, which implies that $\gammaeta^*((E_*,\,\varphi_*))$ is a
		semistable  parabolic Higgs bundle. From this it follows that 
		$(E_*,\,\varphi_*)$ must be parabolic semistable as well, since any $\varphi_*$-invariant
		sub-bundle $E'_*\,\subset\, E_*$ of strictly larger parabolic slope would give rise to a $\gammaeta^*(\varphi_*)$-invariant
		sub-bundle $\gammaeta^*(E'_*)\,\subset\, \gammaeta^*(E_*)$ of strictly larger parabolic slope, contradicting the
		parabolic semistability of $\gammaeta^*((E_*,\,\varphi_*))$. This proves the claim that $(E_*,\,\varphi_*)$ is parabolic semistable.
		
		\textit{Moreover, $(E_*,\,\varphi_*)$ is actually parabolic stable}. To see this, take any non-trivial $\varphi_*$-invariant
		sub-bundle $E'_*\,\subset\, E_*$ such that $\textnormal{par}\mu(E'_*)\,=\, \textnormal{par}\mu(E_*)$. Then
		$$\gammaeta^*(E'_*)\,\subset\, \gammaeta^*(E_*)\,=
		\,\underset{\sigma\in \textnormal{Gal}(\gammaeta)}{\bigoplus} \sigma^*(F_*)$$
		is a $\gammaeta^*(\varphi_*)$-invariant sub-bundle with same parabolic slope. The parabolic
		Higgs field on $\gammaeta^*(E'_*)$ induced by $\gammaeta^*(\varphi_*)$ will also be denoted by $\gammaeta^*(\varphi_*)$.
		
		Now, it is clear that $(\gammaeta^*(E'_*),\,\gammaeta^*(\varphi_*))$ is also parabolic semistable. 
		Thus, it contains a parabolic stable Higgs bundle $F'_*$ equipped with parabolic Higgs field
		induced by $\gammaeta^*(\varphi_*)$, and having the same parabolic slope as $\gamma_\eta^*(E_*')$. Next note that the parabolic Higgs bundles 
		$$\{\sigma^*(F_*,\,\phi_*)\,\mid\,\sigma\in\textnormal{Gal}(\gammaeta)\}$$ are mutually 
		non-isomorphic. Indeed, otherwise there would exist a parabolic isomorphism $\widetilde f$ 
		from $(F_*,\,\phi_*)$ to $\sigma^*(F_*,\,\phi_*)$ for some $\sigma\,\in\, {\rm 
			Gal}(\gamma)\setminus\{e\}$. Such an isomorphism would preserve the filtrations of $F_y$ 
		and $F_{\sigma(y)}$ at a parabolic point $y$. But since the weights of $F_*$ at the 
		parabolic points $y$ and $\sigma(y)$ are collectively distinct, and a parabolic 
		isomorphism preserves weights, this leads to a contradiction. We thus see that all 
		projections $(F'_*,\,\gammaeta^*(\varphi_*))\,\longrightarrow\, \sigma^*(F_*,\phi_*)$ except one
		$\sigma$ must be zero, so we have $F'_*\,=\, \sigma^*(F_*,\phi_*)$ for some $\sigma$.
		Since $\gammaeta^*(E'_*)$ is $\textnormal{Gal}(\gammaeta)$-equivariant, and it contains
		$F'_*\,=\, \sigma^*(F_*,\phi_*)$, we conclude that $\gammaeta^*(E'_*)\,=\, \gammaeta^*(E_*)$. This clearly implies that 
		$E'_*\,=\,E_*$. Thus we have $f((F_*,\,\phi_*))\,\in\,\MHalpha$.
		
		Next, we argue that $\textnormal{Im}(f)\,\subseteq\,\MHetaalpha$. Take any $(E_*,\,\varphi_*)\,=\,f(F_*,\,\phi_*)$. There exists a tautological
		trivialization of the line bundle $\gammaeta^*\eta$ over $Y_\eta$, which induces an isomorphism 
		\begin{align}\label{tautological trivialization}
			\theta\,:\, \mathcal{O}_{Y_\eta}\,\longrightarrow\, \gammaeta^*\eta\,.
		\end{align}  For each $1\,\leq\, i\,\leq\, m$, the map
		$\theta_{q_i}\,:\, \mathbb{C}\,\longrightarrow \,(\gammaeta^*\eta)_{q_i}\,=\,\eta_p$
		(see \eqref{tautological trivialization}) is given by
		$\lambda \,\longmapsto\, \lambda q_i.$ It produces an isomorphism
		\begin{align*}
			{\rm Id}_F\otimes \theta \,:\, F\,\xrightarrow{\,\,\simeq\,\,}\, F\otimes \gammaeta^*\eta.
		\end{align*}
		Taking its direct image, and using the projection formula, the following isomorphism is obtained:
		$$
		\psi\, :=\, \gammaeta_*({\rm Id}_F\otimes\theta)\, :\,\,
		E\,=\, \gammaeta_*F\, \longrightarrow\, \gammaeta_*(F\otimes \gammaeta^*\eta)
		\,=\, (\gammaeta_*F)\otimes \eta\,=\, E\otimes \eta.
		$$
		Now $E_p\,=\, \bigoplus_{i=1}^{m}F_{q_i}$, and the map $\psi_p\,:\,E_p\,
		\longrightarrow\, E_p\otimes \eta_p$ on the fiber takes $F_{q_i}$ to
		$F_{q_i}\otimes \eta_p$, which clearly implies that 
		$\psi_p$ preserves the filtration induced on $E_p$. Thus $\psi$ is a parabolic
		isomorphism. Moreover, $\theta$ in \eqref{tautological trivialization} evidently commutes with the Higgs field $\phi$, which
		shows that $\psi$ commutes with $\gammaeta_*(\phi)$. Consequently, we have
		$\text{Im}(f)\,\subseteq\, \MHetaalpha.$
		
		To prove that $f$ is surjective on $\MHetaalpha$, take any $(E_*,\,\varphi_*)\,\in\, \MHetaalpha$. Recall from
		the earlier discussion that this means that there exists an isomorphism
		\[\widetilde{\psi_*}\,:\, E_*\,\simeq\, E_*\otimes \eta\]
		of parabolic bundles which is compatible with $\varphi_*$, so that the diagram below commutes:
		\begin{align*}\label{commuting diagram for fixed point}
			\xymatrix@=1.8cm{ E_*\ar[r]^{\widetilde{\psi_*}} \ar[d]_{\varphi_*} & E_*\otimes \eta \ar[d]^{\varphi_*
					\otimes{\rm Id}_\eta} \\
				E_*\otimes K_X(D) \ar[r]^(0.45){\widetilde{\psi_*}\otimes{\rm Id}_{K_X(D)}} & E_*\otimes K_X(D)\otimes \eta
			}
		\end{align*}
		
		We need to show that there exists $(F_*,\,\phi_*)
		\,\in\, \mathcal{N}_\eta$ such that $f((F_*,\,\phi_*))\,\simeq\, (E_*,\,\varphi_*)$.
		
		Since $(E_*,\,\varphi_*)$ is a parabolic stable Higgs bundle, it is simple, and hence
		any parabolic endomorphism of $(E_*,\,\varphi_*)$ which commutes with the Higgs field
		is in fact a constant scalar multiplication. As a consequence, any two parabolic isomorphisms from $E_*$ to $E_*\otimes \eta$ which commute with $\varphi_*$ will differ by a constant scalar. Thus, we can re-scale $\widetilde{\psi_*}$ by multiplying with a nonzero scalar, so that the $m$--fold composition 
		\[\underset{m-times}{\widetilde{\psi_*}\,\circ\cdots\circ\,\widetilde{\psi_*}}\,:\, E_*\,\longrightarrow
		\,E_*\otimes \eta^m\]
		coincides with $\textnormal{Id}_{E_*}\otimes s_0$, where $s_0$ is the chosen nowhere-vanishing section of $\eta^m$. This gives
		\[\underset{m-times}{\widetilde{\psi}\,\circ\cdots\circ\,\widetilde{\psi}}\,=\, \textnormal{Id}_E\otimes s_0
		\,:\, E \,\longrightarrow\, E\otimes \eta^m\]
		on the underlying bundles. Then, the argument given in the proof of
		\cite[Lemma 2.1]{BH10} will produce a vector bundle $F$ of rank $l=r/m$ on $Y_\eta$ with
		$\gammaeta_*F\,\cong\, E$. Let us briefly recall the argument for convenience. Consider the pull-back $\gammaeta^*(\widetilde{\psi})$, and compose it with the tautological trivialization of $\gammaeta^*\eta$ to get a morphism 
		\[\widetilde{\phi}\,:\, \gammaeta^*E\,\longrightarrow \,\gammaeta^*E.\] 
		Since $Y_\eta$ is irreducible, the characteristic polynomial of $\widetilde{\phi}_y$ remains unchanged
		as the point $y\,\in\, Y_\eta$ moves. This allows us to decompose $\gammaeta^*E$ into
		generalized eigenspace sub-bundles. If $F$ is an eigenspace sub-bundle of $E$, then
		we have $\gammaeta_*F\,\cong\, E$ due to the observation that the eigenvalues are the powers of the root of unity $\mu$ and the action of the element $\mu^i$ sends the $\lambda$-eigenspace sub-bundle to $\mu^i\lambda$-eigenspace sub-bundle.
		
		Our next task is to produce a parabolic structure on $F$ and a Higgs field $\phi_*$ on the resulting parabolic
		bundle $F_*$ so that the parabolic pushforward $(\gammaeta_*(F_*),\gammaeta_*(\phi_*))$ coincides with $(E_*,\varphi_*)$.
		
		First we produce the parabolic structure. Recall the description of $\theta$ in (\ref{tautological trivialization}), and notice that for any choice of $q\in \gammaeta^{-1}(p)$, the map $\widetilde{\phi}_q$ is precisely the composition
		\begin{align}\label{phi_q}
			(\gammaeta^*E)_q\,=\,E_p\,\xrightarrow{\,\widetilde{\psi}_p\,}\, E_p\otimes \eta_p \,=\, (\gammaeta^*E)_q\otimes
			(\gammaeta^*\eta)_q\,\xrightarrow{\,\,Id\otimes(\theta_q)^{-1}\,} \, (\gammaeta^*E)_q\,,
		\end{align}
		where $\theta_q\,:\, \mathbb{C}\,\longrightarrow \,(\gammaeta^*\eta)_q\,=\,\eta_p$ is
		defined by $\lambda\,\longmapsto\, \lambda q$. Thus, if
		\begin{align*}
			E_p\,=\,E^p_1\,\supsetneq \,E^p_2\,\supsetneq\,\cdots\,\supsetneq\,E^p_r\,\supsetneq\,0
		\end{align*}
		is the given parabolic filtration of the fiber $E_p$, then as $\varphi_*$ is a parabolic isomorphism,
		\begin{align}
			\forall\,j\,\in\,[0,\,r],\,\,\,\{\varphi_p(E^p_j)\,=\,E^p_j\otimes \eta_p\}\,\implies\,\{\phi_q(E^p_j)\,=\,E^p_j\}\,\,\,\,[\text{from}\,(\ref{phi_q})].
		\end{align}
		Let $\widetilde{\phi}^s_q$ be the semisimple part of $\widetilde{\phi}_q$ for its Jordan-Chevalley decomposition. It is well-known that $\widetilde{\phi}^s_q$ can be expressed as a polynomial in $\widetilde{\phi}_q$ without constant coefficient. Thus 
		\[\widetilde{\phi}^s_q(E^p_j) = E^p_j\,\,\,\,\,\forall\,\,\,j\,\in\,[0,\,r-1].\]
		Moreover, the generalized eigenspaces of $\widetilde{\phi}_q$ (namely $F_{q_i}$'s) are the eigenspaces for $\widetilde{\phi}^s_q$.
		Thus $\widetilde{\phi}^s_q\,:\,E_p\,\longrightarrow\, E_p$ and the filtration
		$E_p\,=\,E^p_1\,\supsetneq\,\cdots\,\supsetneq \,E^p_r\,\supsetneq\,0$ allow us to apply Lemma \ref{linear algebra result}, which
		gives us a basis of $E_p$\,\,, say $\{v_1,\,\cdots,\,v_r\}$, such that each $E^k_p = \langle v_k,\cdots, v_r\rangle$ and each $v_j$ is contained in a unique $F_{q_l}$. From this data, we can produce a full-flag parabolic structure on the
		fibers $F_{q_1},\cdots,F_{q_m}$ of $F$ as follows:
		Choose a basis $$B\,=\,\{v_1,\,v_2,\,\cdots,\,v_r\}$$ of $E_p$ as in Lemma \ref{linear algebra result}, and for each $q_i$,
		consider the subset $B_i:\,=\,B\cap F_{q_i}\, \subset\, B$. By symmetry, each $B_i$ consists of $l$ elements and spans $F_{q_i}$. Suppose
		$B_i \,=\, \{v_{i_1},\,v_{i_2},\,\cdots,\,v_{i_l}\}$ with $i_1\,<\,i_2\,<\,\cdots\,<\,i_n.$ Then consider the following weighted
		full-flag filtration of $F_{q_i}$:
		\begin{align*}
			F_{q_i}\,=\,\langle v_{i_1},\,v_{i_2},\,\cdots,\,v_{i_l}\rangle &\,\supsetneq\, \langle v_{i_2},\,\cdots,\,v_{i_l}\rangle \,\supsetneq
			\,\cdots \,\supsetneq \,\langle v_{i_n}\rangle\,\supsetneq\, 0,\\
			\alpha^p_{i_1} &\,<\,\alpha^p_{i_2} \,<\,\cdots \cdots\cdots\cdots\,<\,\alpha^p_{i_l},
		\end{align*}
		where $(\alpha^p_1\,<\,\alpha^p_2\,<\,\cdots\,<\,\alpha^p_r)$ are the
		parabolic weights at $p$. By performing this for all $1\,\leq\, i\,\leq\, m$, we get a parabolic bundle $F_*$ on $Y_\eta$.
		
		Next a Higgs field on $F_*$ will be constructed. As it was seen above,
		$\gammaeta^*(E_*)\,\simeq\, \bigoplus_{\sigma\in {\rm Gal}(\gammaeta)}\sigma^*(F_*)$ is the decomposition of $\gammaeta^*(E_*)$ into generalized eigenspace sub-bundles under the pull-back morphism
		\[\gammaeta^*(\psi_*)\,:\,\bigoplus_{\sigma\in {\rm Gal}(\gammaeta)}\sigma^*(F_*)\,\longrightarrow
		\,\bigoplus_{\sigma\in {\rm Gal}(\gammaeta)}(\sigma^*(F_*)\otimes \gammaeta^*(\eta)).\]
		Since $\gammaeta^*(\varphi_*)$ is compatible with $\gammaeta^*(\psi_*)$, it must
		preserve its generalized eigenspaces, and thus $\gammaeta^*(\varphi_*)$ also decomposes as 
		\[\gammaeta^*(\varphi_*)\,=\,\bigoplus_{\sigma\in {\rm Gal}(\gammaeta)}\phi_{\sigma,*}\]
		where $\phi_{\sigma,*}\,:\,\sigma^*(F_*)\,\longrightarrow\,
		\sigma^*(F_*)\otimes K_{\yeta}(\gammaeta^*D)$ are homomorphisms.
		
		Since $\gammaeta^*(\varphi_*)$ is a ${\rm Gal}(\gammaeta)$-equivariant morphism, it follows that $\phi_{\sigma'\sigma,*}$ is
		the conjugate of $\phi_{\sigma,*}$ by the automorphism $\sigma'$. But $\varphi_*$ is the descent of the
		${\rm Gal}(\gammaeta)$-equivariant morphism $\gammaeta^*(\varphi_*)$, which immediately implies by the uniqueness of
		descent that, upon taking $\sigma\,=\,e$ to be identity in ${\rm Gal}(\gammaeta)$, 
		\[\gammaeta_*(\phi_{e,*})\,=\,\varphi_*.\] 
		Moreover, since $\varphi_*$ is a  parabolic morphism by assumption, it forces each $\phi_{\sigma,*}$ to be
		parabolic as well, due to the nature of its construction from \S~\ref{parabolic pushforward}. Thus $\phi_*\,:= \,
		\phi_{e,*}$ is our candidate for the Higgs field.
		
		Note that $F_*$ must be parabolic stable, because if $F'\,\subset\, F$ is any $\phi$-invariant sub-bundle such that $\text{Par}\mu(F'_*)
		\,\geq\, \text{Par}\mu(F_*)$, then the equalities
		\[\text{Par}deg(\gammaeta_*(F'_*))\,=\,\text{Par}deg(F_*)\,\,\,\,\text{ and }\,\,\,\,\text{rank}(\gammaeta_*(F')) \,=\, m\cdot \text{rank}(F')\]
		would imply that $\gammaeta_*(F')\,\subset\, E$ violates the condition of parabolic stability for $E_*$. Thus $F_*\,\in\,
		\mathcal{N}^\textbf{t}_\eta$ for some $\textbf{t}\,\in\,  {\bf P}(\alpha).$ 
		
		Finally, if the number of parabolic points on $X$ is more than 1, then an exactly similar argument with the obvious modifications will give the result.
	\end{proof}
	
	\begin{corollary}\label{quotient}
		Each fiber of  $f$ coincides with an orbit of ${\rm Gal}({\gammaeta})$. Therefore, via $f$,
		$$\mathcal{N}_\eta/{\rm Gal}({\gammaeta})\,\simeq\, \MHetaalpha.$$
	\end{corollary}
	\begin{proof}
		 With the same notations as in Lemma \ref{parabolic fixed point}, suppose $(F_*\,,\,\phi_*)$ and $(F'_*\,,\,\phi'_*)$ in $\mathcal{N}_\eta := \underset{t\in  {\bf P}(\alpha)}{\coprod}\mathcal{N}^\textbf{t}_\eta$\, satisfy $f((F_*,\,\phi_*))\,=\,f((F'_*,\,\phi'_*))$.  Taking pull-back by $\gamma_{ \eta}$ would then imply 
		 \begin{align*}
		 	\bigoplus_{\sigma\in {\rm Gal}(\gammaeta)}\sigma^*(F_*\ ,\ \phi_*)\simeq \bigoplus_{\tau\in{\rm Gal}(\gammaeta)}\tau^*(F'_*\ ,\ \phi'_*)\ .
		 \end{align*}
		 As the direct summands on both sides are parabolic stable, this clearly implies that $(F'_*\ ,\ \phi'_*)\simeq \sigma^*(F_*\ ,\ \phi_*)$ for some $\sigma\in\rm{Gal}(\gammaeta)$. Also, recall that in the proof of Lemma \ref{parabolic fixed point} we have already shown that the collection of parabolic Higgs bundles $\{\sigma^*(F_*\ ,\ \phi_*)\ \mid\ \sigma\in\rm{Gal}(\gammaeta)\}$ are mutually non-isomorphic (cf. page 9). Thus, it follows that $(F'_*\ ,\ \phi'_*)\simeq \sigma^*(F_*\ ,\ \phi_*)$ for a unique $\sigma\in\rm{Gal}(\gammaeta)$. Thus the group $\rm{Gal}(\gammaeta)$ acts freely and transitively on the fibers of the surjective map $f: \mathcal{N}_{\eta}\rightarrow\MHetaalpha$, from which we conclude that
		 $$\mathcal{N}_\eta/{\rm Gal}({\gammaeta})\,\simeq\, \MHetaalpha\ .$$
		 This proves our claim.
	\end{proof}
	
	\section{Connected components of fixed point loci}\label{connected componets of fixed point loci}
	
	In this section the connected components of $\MHetaalpha$ will
	be described. First, let us define an action of ${\rm Gal}({\gammaeta})$ on the set $ {\bf P}(\alpha)$ described in \S~\ref{estimate}.
	
	Fix a generator $\mu$ of ${\rm Gal}({\gammaeta})$, and also fix an ordering of ${\gammaeta}^{-1}(p)$ for each $p\,\in\,
	S$. With this ordering, write $${\gammaeta}^{-1}(p) \,=\, \{q_{1,p},\,q_{2,p},\,\cdots,\, q_{m,p}\}\,\,\,\text{such that}\,\,
	\mu^i(q_{j,p})\,=\, q_{i+j,p}\, ,$$
	where $(i+j)$ is to be understood mod $m$. Each $\textbf{t}\,\in\,  {\bf P}(\alpha)$ then determines an ordered
	partition of $\{\alpha^p_1,\, \alpha^p_2,\,\cdots, \,\alpha^p_r\}$ each having cardinality $r/m$, namely
	$\{\textbf{t}^p_{1},\,\textbf{t}^p_{2},\,\cdots ,\,\textbf{t}^p_{m}\}$ for each $p\,\in\, S$. Moreover, by
	definition, $\textbf{t}^p_j$ describes the full-flag parabolic data at $q_{j,p}$. For
	$\mu^i\,\in\, {\rm Gal}({\gammaeta})$
	, define $\mu^i\cdot \textbf{t}$ to be the new ordered partition
	$$\{\textbf{t}^p_{i+1},\,\textbf{t}^p_{i+2},\,\cdots ,\,\textbf{t}^p_{i+m}\}_{p\in S}\,\,,$$ 
	where, now, $\textbf{t}^p_{i+j}$ describes the parabolic structure at $q_{j,p}$ for all $j\,\in\,[1,\,m]$. It is easy to check that this action of ${\rm Gal}({\gammaeta})$ is free due to the condition that all the weights are different.
	
	Now, ${\rm Gal}({\gammaeta})$ acts on $\mathcal{N}_\eta\,=\,
	\underset{\textbf{t}\in  {\bf P}(\alpha)}{\coprod}\mathcal{N}^\textbf{t}_\eta$ by pull-back. Clearly
	$$\{\sigma\,\in\, {\rm Gal}({\gammaeta}),\,\,(F_*,\,\phi_*)\,\in\, \mathcal{N}^\textbf{t}_{\eta}\}\,\,\implies\,\,
	\{(\sigma^*(F_*),\,\sigma^*(\phi_*))\,\in\, \mathcal{N}^{\sigma\cdot \textbf{t}}_{\eta}\}.$$
	In fact, we have an isomorphism $\mathcal{N}^\textbf{t}_\eta \,\simeq\, \mathcal{N}^{\sigma\cdot \textbf{t}}_\eta$
	defined by $(F_*,\,\phi_*)\,\longmapsto\, (\sigma^*(F_*),\,\sigma^*(\phi_*))$.
	
	\begin{lemma}\label{components}
		The map 
		$$f\,:\, \mathcal{N}_\eta\,=\,\underset{t\in  {\bf P}(\alpha)}{\coprod}\mathcal{N}^\textbf{t}_\eta\,\longrightarrow\, \MHetaalpha$$ 
		given by parabolic pushforward satisfies the following properties:
		\begin{enumerate}[(i)]
			\item For each $\textbf{t}\,\in \, {\bf P}(\alpha)$, the nap $f^\textbf{t}\,:=\,f|_{\mathcal{N}^\textbf{t}_\eta}$ is injective; thus
			$\mathcal{N}^\textbf{t}_\eta\,\simeq\, f(\mathcal{N}^\textbf{t}_\eta)$ can be considered as a closed (and open) subset of $\Metaalpha$\,.\\
			
			\item For $\textbf{t},\,\textbf{t}'\,\in\,  {\bf P}(\alpha)$, the following are equivalent:
			\begin{enumerate}
				\item $\textbf{t}'\,=\,\sigma\cdot \textbf{t}$ for some $\sigma\,\in\, {\rm Gal}({\gammaeta})$,
				
				\item $f(\mathcal{N}^\textbf{t}_\eta)\,=\,f(\mathcal{N}^{\textbf{t}'}_\eta)$,
				
				\item $f(\mathcal{N}^\textbf{t}_\eta)\cap f(\mathcal{N}^{\textbf{t}'}_\eta)\,\neq\, \emptyset$.
			\end{enumerate}
		\end{enumerate}
	\end{lemma}
	
	\begin{proof}
		$(i):$\, This follows from Corollary \ref{quotient}, which says that a fiber of $f$ coincides with a ${\rm
			Gal}({\gammaeta})$-orbit, together with the fact that $\mathcal{N}^\textbf{t}_\eta\,\neq\, \sigma^*(\mathcal{N}^\textbf{t}_\eta )$
		for a non-trivial $\sigma\,\in \,{\rm Gal}({\gammaeta})$. Since the quotient by a finite group is a proper map, it follows that
		$f(\mathcal{N}^\textbf{t}_\eta)$ is closed. The fact that $f(\mathcal{N}^{\bf t}_\eta)$ is also open in $\MHetaalpha$ follows
		from $(ii)$ below.
		
		\begin{enumerate}[]
			\item $(ii)$:\, $(a)\implies (b)$:\, We saw that $(F_*,\,\phi_*)\,\in\, \mathcal{N}^\textbf{t}_{\eta}$
			if and only if $(\sigma^*(F_*),\,\sigma^*(\phi_*))\,\in \,\mathcal{N}^{\sigma\cdot \textbf{t}}_{\eta}$. Since
			${\gammaeta}_*((F_*,\,\phi_*))\,=\,{\gammaeta}_*((\sigma^*(F_*),\,\sigma^*(\phi_*)))$, the claim follows.
			
			\item $(b)\implies (c)$:\, This is obvious.
			
			\item $(c)\implies (a)$:\, The fibers of $f$ coincide with ${\rm Gal}({\gammaeta})$--orbits (see Corollary \ref{quotient}). Thus,
			if some fiber of $f$ intersects both $\mathcal{N}^\textbf{t}_\eta$ and $\mathcal{N}^{\textbf{t}'}_\eta$, then
			there exists $(F_*,\,\phi_*)\,\in\,\mathcal{N}^\textbf{t}_\eta$ such that $(\sigma^*(F_*),\,\sigma^*(\phi_*))
			\,\in \,\mathcal{N}^{\textbf{t}'}_{\eta}$ for some $\sigma\,\in\, {\rm Gal}({\gammaeta})$. This
			implies that $\textbf{t}'\,=\,\sigma\cdot \textbf{t}$. 
		\end{enumerate}
		This completes the proof.
	\end{proof}
	
	{}From Lemma \ref{components} it is evident that in order to understand the connected components of $\MHetaalpha$, it is
	enough to understand the connected components of $\mathcal{N}^\textbf{t}_\eta$ for various $\textbf{t}\,\in\, P(\alpha)$.
	
	As we have $\gammaeta^*(K_X(D))\,=\,K_{\yeta}(\gammaeta^*D)$, it follows that $K_{\yeta}(\gammaeta^*D)$ is naturally a
	${\rm Gal}(\gammaeta)$-equivariant bundle. Thus ${\rm Gal}(\gammaeta)$ also acts on $H^0(\yeta,\,K_{\yeta}(\gamma^*D))$; an element $\sigma
	\,\in\, {\rm Gal}(\gammaeta)$ acts by sending a section $s\,\longmapsto\, \sigma s \sigma^{-1}$ via the diagram
	\begin{align}\label{diagram 0}
		\xymatrix{\gammaeta^*(K_X(D))\ar[r]^{\cdot \sigma} \ar[d] & \gammaeta^*(K_X(D)) \ar[d] \\
			\yeta \ar[r]^{\cdot \sigma} & \yeta
		}
	\end{align}
	Now, we have 
	\begin{align}
		H^0(\yeta,\, K_{\yeta}(\gammaeta^*D)) \,=\, H^0(\yeta,\, \gammaeta^*(K_X(D))) &\,=\, H^0(X,\,\gammaeta_*(\gammaeta^*(K_X(D))))\\
		&\simeq\, H^0(X,\,K_X(D)\otimes\gammaeta_*\mathcal{O}_{\yeta})\,\,\,\,[\text{projection formula}]\\
		& \simeq\, \bigoplus_{i=0}^{m-1} H^0(X,\,K_X(D)\otimes \eta^i)\,\,\,[\text{as }\, \gammaeta_*\mathcal{O}_{\yeta}
		\,\simeq \,\bigoplus_{i=0}^{m-1}\eta^i].
	\end{align}
	Moreover, the decomposition $H^0(\yeta,\, K_{\yeta}(\gammaeta^*D))\,\simeq\,
	\bigoplus_{i=0}^{m-1} H^0(X,\,K_X(D)\otimes \eta^i)$ is in fact the eigenspace decomposition of
	$H^0(\yeta ,\, K_{\yeta}(\gammaeta^*D))$ under the action of the cyclic group ${\rm Gal}(\gammaeta)$ as defined above
	\cite[Corollary 3.11]{EV92}.
	Thus, $H^0(X,\,K_X(D))$ can be identified with the subspace of $H^0\left(\yeta,\, K_{\yeta}(\gammaeta^*D)\right)$ corresponding
	to the eigenvalue $1$; in other words, it is the subspace of ${\rm Gal}(\gamma)$--fixed points. Consequently,
	we get a natural linear surjection 
	\begin{align}\label{linear map 1}
		h\,:\, H^0\left(\yeta,\, K_{\yeta}(\gammaeta^*D)\right)&\,\longrightarrow\, H^0(X,\,K_X(D))
	\end{align}
	that sends a section $s$ to the descent of the ${\rm Gal}(\gammaeta)$-invariant section $\underset{\sigma\,\in
		\,{\rm Gal}(\gammaeta)}{\sum}\sigma s \sigma^{-1}$ (see (\ref{diagram 0})).
	
	For future use, let us note that each fiber of $h$ is connected.
	
	\begin{lemma}[\text{\cite[Lemma 3.18]{Na05}}]\label{lem3}
		Let $\mu_m$ denote the group of $m$-th roots of unity in ${\mathbb C}^*$. Consider the left regular representation of $\mu_m$, namely
		$\mu_m$ acting on the group algebra $k[\mu_m]$ by left multiplication. Let $\chi_{\text{reg}}
		\,:\,\mu_m\,\longrightarrow
		{\mathbb C}^*$ be the determinant of the regular representation, and let $L_{\chi_{\text{reg}}}\,:=\,
		\yeta\times^{\chi_{\text{reg}}}\mathbb C$ denote the associated line bundle on $X$. Then
		\begin{equation*}
			L_{\chi_{\text{reg}}} \simeq
			\begin{cases}
				\mathcal{O}_X & {\rm if }\, m\,\,\textnormal{is odd,}\\
				{\eta^{m/2}} & \text{if}\ m\ \text{is even}\ .
			\end{cases}
		\end{equation*}
	\end{lemma}
	
	\begin{proof}
		Let $\mu_m\,=\,\langle \xi_m\rangle$. From the description of the tautological section of $\gammaeta^*(\eta)
		\,\simeq\, \mathcal{O}_{\yeta}$, it is straightforward to see that if $\chi_0 \,:\, \mu_m\,\hookrightarrow\, {\mathbb C}^*$
		denotes the inclusion morphism, then $\eta$ is isomorphic to the associated line bundle $Y_{\chi_0}
		\,=\, \yeta\times^{\chi_0}\mathbb C$. Choosing an ordered basis $\{1,\,\xi_m,\,\xi_m^2,\,\cdots,\,\xi_m^{m-1}\}$ we see that
		\begin{align*}
			\chi_{\text{reg}}(\xi_m) \,=\, \det\begin{bmatrix}
				0 & 0 & \cdots & 0 & 1 \\
				1 & 0 & \cdots & 0 & 0 \\
				\vdots & \vdots & \cdots & \vdots & \vdots \\
				0&0&\cdots&1&0
			\end{bmatrix} \,=\, (-1)^{m-1}.
		\end{align*} 
		Thus for $m$ odd, $\chi_{\text{reg}}$ is trivial, whereas for $m$ even, $\chi_{\text{reg}} \,=\, (\chi_0)^{m/2}$, and
		consequently $L_{\chi_{\text{reg}}}\,\simeq\, L_{\chi_0}^{m/2} \,\simeq\, \eta^{m/2}$. 
	\end{proof}
	
	\begin{lemma}[\text{\cite[Proposition 3.50]{Na05}}]\label{lem4}
		Let $F$ be a vector bundle of rank $r/m$ on $\yeta$. Let $\gammaeta\,:\,\yeta\,\longrightarrow\, X$ be the above spectral curve. Then 
		\begin{equation*}
			{\det}_{X}(\gammaeta_*(F)) \simeq
			\begin{cases}
				Nm_{\gammaeta}(\det(F)) & {\rm if }\, m\,\,\textnormal{is odd,}\\
				Nm_{\gammaeta}(\det(F))\otimes \eta^{r/2} & \text{if}\ m\ \text{is even}\ .
			\end{cases}
		\end{equation*}
	\end{lemma}
	
	\begin{proof}
		Let us denote $G\,=\,{\rm Gal}(\gammaeta)$. We have $G\,\simeq\, \mu_m\,\subset\,{\mathbb C}^*$, the group of $m$th roots of unity. The bundle $\gammaeta_*(E)$ is obtained as the descent of the $G$--equivariant bundle $\bigoplus_{\sigma\in G}\sigma^*E$. We have
		\begin{align}\label{eq2} 
			\det(\bigoplus_{\sigma\in G}\sigma^*F) \simeq \underset{\sigma\in G}{\otimes}\sigma^*(\det F) \simeq (\underset{\sigma\in G}{\otimes}\sigma^*\det F) \otimes \mathcal{O}_{\yeta}\,.
		\end{align}
		We want to view this isomorphism as an equivariant isomorphism. On $\underset{\sigma\in G}{\otimes}\sigma^*F$, $G$ acts naturally by permuting the summands, together with moving the fibers. If we identify the fibers of $\underset{\sigma\in G}{\oplus}\sigma^*F$ with the direct sum of $\frac{r}{m}$ copies of the group ring $k[\mu_m]^{\oplus r/m}$, we see immediately that the $G$--action on fibers of (\ref{eq2}) coincides with direct sum of $r/m$ copies of regular representation of $G$. Thus, on the left-hand side of (\ref{eq2}), the $G$--action is precisely the $\frac{r}{m}$'th power of the determinant of the regular representation. Since the $G$--action on the fibers of $\underset{\sigma\in G}{\otimes}\sigma^*\det F$ is trivial, it follows immediately that in order for the isomorphism (\ref{eq2}) to be equivariant, $G$ acts on $\mathcal{O}_{\yeta}$ via the character $\chi_{\text{reg}}^{\frac{r}{m}}$\,, where $\chi_{\text{reg}}$ denotes the determinant of the regular representation of $\mu_m$.\\
		Since descent commutes with tensor products of equivariant bundles, it follows that the descents of (\ref{eq2}) are isomorphic. The bundle $\det(\bigoplus_{\sigma\in G}\sigma^*F)$ descends to $\det(\gammaeta_*F)$, the bundle $\underset{\sigma\in G}{\otimes}\sigma^*\det F$ descends to $Nm_{\gammaeta}(\det F)$, and the bundle $\mathcal{O}_{\yeta}$ with action via $\chi_{\text{reg}}^{r/m}$ descends to $\mathcal{O}_X$ if $m$ is odd, and to $(\eta^{m/2})^{r/m} = \eta^{r/2}$ if $m$ is even [Lemma \ref{lem3}]. The result follows.
	\end{proof}
	
	\begin{lemma}\label{lem1}
		Let $l=r/m$. With the same notations as above, for all $\textbf{t}\in  {\bf P}(\alpha)$, the following diagram commutes:
		\begin{align}\label{diagram 1}
			\xymatrix@=2cm{
				M^\textbf{t}_{Y_\eta,H}(l,d) \ar[r]^{f^\textbf{t}} \ar[d]_{(\det_{\yeta}\,,\, tr_{\yeta})} & M^{\alpha}_{X,H}(r,d) \ar[d]^{(\det_X\,,\, tr_X)}\\
				\Pic^d(Y_\eta)\times H^0(\yeta,K_{\yeta}(\gammaeta^*D)) \ar[r]^{g
					\,\times\, h} & \Pic^d(X)\times H^0(X,K_X(D))
			}
		\end{align}
		where $f^{\textbf{t}}$ is the map from Lemma \ref{parabolic fixed point},\ $h$ is as in \eqref{linear map 1}, and $g:\Pic^d(Y_\eta)\rightarrow \Pic^d(X)$ is given by
		\begin{equation}
			g(L) =
			\begin{cases}
				Nm_{{\gammaeta}}(L)\ , &  \text{if}\ m \ \text{is odd}\\
				Nm_{{\gammaeta}}(L)\otimes \eta^{r/2}\ , & \text{if}\ m\ \text{is even}\ .
			\end{cases}
		\end{equation}
	\end{lemma}
	\begin{proof}
		Let $(F_*\,,\,\phi_*)\in M^\textbf{t}_{\yeta,H}(l,d)$. By definition, $f^\textbf{t}(F_*\,,\,\phi_*)=(\gammaeta_*(F_*)\,,\,\gammaeta_*(\phi_*))$. The underlying bundle $F$ satisfies ${\det}_X(\gammaeta_*(F))=g({\det}_{\yeta}(F))\,$ by Lemma \ref{lem4}.	The fact that $tr_X(\gammaeta_*(\phi)) = h(tr_{\yeta}(\phi))$ also follows from the description of $h$.
	\end{proof}
	Henceforth, we fix  $l=r/m$. 
	\begin{proposition}\label{inverse of connected under det is connected}
		Let $g$ and $h$ be as in Lemma \ref{lem1}. For any $\xi\in\Pic^d(X)$ and $\omega\in H^0(K_X(D))$\,, $(g\times h)^{-1}(\xi,\omega)$ has $m$ connected components. Moreover, $(\det_{\yeta},tr_{\yeta})^{-1}(Z)$ is connected for any connected component $Z$ of $ (g\times h)^{-1}(\xi,\omega)$. As a consequence, for any $\textbf{t}\in  {\bf P}(\alpha)$, $\mathcal{N}^\textbf{t}_\eta$ also has $m$ connected components, namely 
		$$\{({\det}_{\yeta},tr_{\yeta})^{-1}(C\times h^{-1}(0))\,\mid\,C\,\,\text{is a connected component of}\,\,g^{-1}(\xi)\}\,\,.$$ 
	\end{proposition}
	\begin{proof}
		Since $(g\times h)^{-1}(\xi,\omega) = g^{-1}(\xi)\times h^{-1}(\omega)$ and clearly $h^{-1}(\omega)$ is connected, the first assertion follows from the fact that $g^{-1}(\xi)$ has $m$ connected components \cite[Proposition 5.7]{GO18} (see also \cite[Lemma 3.4]{NR75}).\\
		For any $\textbf{t} \in  {\bf P}(\alpha)$, consider the morphism $$({\det}_{Y_\eta},tr_{\yeta}):M^\textbf{t}_{\yeta,H}(l,d)\rightarrow \Pic^d(Y_\eta)\times H^0(\yeta,K_{\yeta}(\gammaeta^*D))\,. $$
		We claim that the map $(\det_{Y_\eta},tr_{\yeta})$ is \'etale-locally trivial. To prove this,
		choose any fixed line bundle $\mathcal{L}_{\eta}$ of degree $d$ on $\yeta$, and define a map 
		\begin{align}\label{eqn:mu-l-map}
			\mu_l : \Pic^0(\yeta) &\rightarrow \Pic^d(\yeta)\\
			L&\mapsto L^l\otimes \mathcal{L}_{\eta}
		\end{align} 
		Next, consider the \'etale cover 
		\begin{align*}
			\mu_l\times \text{Id}:\Pic^0(\yeta)\times H^0(\yeta,K_{\yeta}(\gammaeta^*D))&\rightarrow Pic^d(\yeta)\times H^0(\yeta,K_{\yeta}(\gammaeta^*D))\\
			(L\,,\,\omega)&\mapsto (L^l\otimes \mathcal{L}_{\eta}\,,\,\omega),
		\end{align*}
		By definition, the fiber product
		$$(\mu_l\times \text{Id})^*\left(M^\textbf{t}_{\yeta,H}(l,d)\right) = \{((F_*\,,\,\phi_*)\,,\,(L\,,\,\omega))\,\mid\, \det(F)\simeq L^l\otimes \mathcal{L}_{\eta}\,\,\text{and}\,\,tr(\phi)=\omega\}$$
		From this description, it is easy to see that
		\begin{align}\label{pullback description}
			(\mu_l\times \text{Id})^*\left(M^\textbf{t}_{\yeta,H}(l,d)\right)&\simeq M^\textbf{t}_{\yeta,H}(l,\mathcal{L}_{\eta})\underset{spec(k)}{\times}\left(\Pic^0(\yeta)\times H^0(\yeta,K_{\yeta}(\gammaeta^*D))\right)\,,\\
			\text{under the map}\,\,((F_*\,,\,\phi_*)\,,\,(L\,,\,\omega))&\mapsto ((F_*\otimes L^{\vee}\,,\,\phi_*\otimes Id_{L^{\vee}}-Id_{F\otimes L^{\vee}}\otimes\omega)\,,\,(L\,,\,\omega))\,.
		\end{align}
		This proves our claim.	Moreover, for generic weights, $M^\textbf{t}_{\yeta,H}(l,\mathcal{L}_{\eta})$ is a smooth irreducible quasi-projective variety. Thus, $({\det}_{\yeta}\,,\,tr_{\yeta}):M^\textbf{t}_{\yeta,H}(l,d)\rightarrow \Pic^d(\yeta)\times H^0(\yeta,K_{\yeta}(\gammaeta^*D))$ has connected fibers, and it is flat and finite-type implies that $({\det}_{\yeta}\,,\,tr_{\yeta})$ is universally open. Now, clearly 
		$$\mathcal{N}^\textbf{t}_{\eta}= (f^\textbf{t})^{-1}\left(({\det}_X,tr_X)^{-1}(\xi,0)\right)=({\det}_{Y_\eta},tr_{\yeta})^{-1}\left((g\times h)^{-1}(\xi,0)\right),$$
		which implies that $(\det_{Y_\eta},tr_{\yeta}): \mathcal{N}^\textbf{t}_{\eta}\rightarrow (g\times h)^{-1}(\xi,0)$ is an open map with connected fibers. The assertions now follow from an exactly similar argument as in \cite[\href{https://stacks.math.columbia.edu/tag/0377}{Tag 0377}\,,\,Lemma 5.7.5]{stacks-project}.\\
		Finally, the last assertion follows since the fibers of $h$ are connected, so that the connected components of $(g\times h)^{-1}(\xi,0)$ are precisely of the form $C\times h^{-1}(0)$, where $C$ is a component of $g^{-1}(\xi)$\,.
	\end{proof}
	\begin{corollary}\label{number of components}
		Let $l=r/m$ and $s=|S|$. For any $\eta\in \Gamma$ with $ord(\eta)= m$, the number of connected components of the fixed point locus $\MHetaalpha$ is equal to $| {\bf P}(\alpha)|=\left({r\choose l}{r-l\choose l}\cdots{l\choose l}\right)^s = \left(\dfrac{r!}{(l!)^{m}}\right)^s$.
	\end{corollary}
	\begin{proof}
		By Proposition \ref{inverse of connected under det is connected}, the number of connected components of $\mathcal{N}^t_\eta$ is $m$ for any $\textbf{t}\in  {\bf P}(\alpha)$. Also, we saw in Lemma \ref{components} that the elements of $ {\bf P}(\alpha)$ in the same ${\rm Gal}({\gammaeta})$-orbit have the same image under $f$, and elements in the different ${\rm Gal}({\gammaeta})$-orbits have disjoint image. Since ${\rm Gal}({\gammaeta})$ acts freely on $ {\bf P}(\alpha)$, it follows that if $\pi_0(\mathcal{N}^\textbf{t}_\eta)$ denotes the set of connected components of $\mathcal{N}^\textbf{t}_\eta$, then the number of connected components of $\MHetaalpha$ equals $$\dfrac{| {\bf P}(\alpha)|}{|{\rm Gal}({\gammaeta})|}\cdot |\pi_0(\mathcal{N}^\textbf{t}_\eta)|\,.$$ 
		But $|{\rm Gal}({\gammaeta})|=|\pi_0(\mathcal{N}^\textbf{t}_\eta)|=m$ [Proposition \ref{inverse of connected under det is connected}], and $| {\bf P}(\alpha)|\, = \, \left({r\choose l}{r-l\choose l}\cdots{l\choose l}\right)^s = \left(\dfrac{r!}{(l!)^{m}}\right)^s$ as seen earlier.
	\end{proof}
	
	Next we describe how the action of $\Gamma$ permutes the set of connected components of $\MHetaalpha$ under certain conditions. This will be used in Section \ref{chen-ruan cohomology section} in  the computation of degree-shift numbers.
	
	\subsection{An action of $\Gamma$ on $\mathcal{N}^\textbf{t}_{\eta}$}\label{gamma action subsection}
	Recall that, by definition, 
	$$\mathcal{N}^\textbf{t}_\eta=\{(F_*\,,\,\phi_*)\,\mid\, \det ({\gammaeta}_*F)\simeq \xi\,\,\text{and}\,\,tr(\gammaeta_*\phi)=0\}.$$
	As in \S\ref{estimate}, for each $\textbf{t}\in  {\bf P}(\alpha)$ let $M^\textbf{t}_{\yeta,H}(n,d)$ denote the moduli of stable parabolic Higgs bundles on $\yeta$ of rank $n$, degree $d$ and having full-flag parabolic structures at the points of $\gammaeta^{-1}(p)$ according to $\textbf{t}$.\\
	First, note that $\Gamma$ acts on $M^\textbf{t}_{\yeta,H}(l,d)$ : an element $\delta\in \Gamma$ acts on $(F_*\,,\,\phi_*)\in M^\textbf{t}_{\yeta,H}(l,d)$ by
	$$(F_*\,,\,\phi_*)\,\mapsto (F_*\otimes {\gammaeta}^*\delta\,,\,\phi_*\otimes Id_{{\gammaeta}^*\delta})\,.$$
	We claim that this action keeps $\mathcal{N}^\textbf{t}_\eta$ invariant.
	This is because, for $(F_*\,,\,\phi_*)\in \mathcal{N}^\textbf{t}_\eta$ and $\delta\in \Gamma$,
	\begin{align*}
		\det({\gammaeta}_*(F\otimes {\gammaeta}^*\delta)) &= \det ({\gammaeta}_*(F)\otimes\delta)\,\,\,\,[\text{projection formula}]\\
		&= \det{\gammaeta}_*(F)\otimes \delta^r\\
		& \simeq \xi 
	\end{align*}
	Moreover, $tr(\gammaeta_*(\phi_*\otimes Id_{\gammaeta^*\delta})) = tr(\gammaeta_*(\phi_*)\otimes Id_\delta) = 0$. Therefore, the action of $\Gamma$ keeps $\mathcal{N}^\textbf{t}_\eta$ invariant, and induces a $\Gamma$-\,action on $\mathcal{N}^\textbf{t}_\eta$\,. \\
	This, in turn, induces an action of $\Gamma$ on its set of connected components, namely $\pi_0(\mathcal{N}^{\textbf{t}}_{\eta})$\,.\\
	
	The next proposition is a generalization of \cite[Proposition 5.13]{GO18}.
	\begin{proposition}\label{Gamma action on components}
		Let $\textbf{t}\in  {\bf P}(\alpha)$ . If $gcd(l,m)=1$ (for example if $r$ is a product of distinct primes), then there exists a line bundle $\delta_{\eta}\in \Gamma$ (depending on $\eta$) of order $m$ with the property that the subgroup $\langle \delta_\eta\rangle\simeq \mathbb{Z}/m\mathbb{Z}$ acts freely and transitively on $\pi_0(\mathcal{N}^\textbf{t}_\eta)$\,\,, while any element of\, $\Gamma$ not in $ \langle \delta_\eta\rangle$ acts trivially on $\pi_0(\mathcal{N}^\textbf{t}_\eta)$\,.
	\end{proposition}
	\begin{proof}
		We know that the $m$--torsion points of $\Pic^0(X)$, denoted by $\Pic^0(X)[m]$, act on $Nm_{\gammaeta}^{-1}(\xi)$. An element $\delta\in \Pic^0(X)[m]$ acts on $L\in Nm^{-1}_{\gammaeta}(\xi)$ by
		$$L\mapsto L\otimes {\gammaeta}^*\delta.$$ 
		Recall the map $g:\Pic^0(Y_\eta)\rightarrow \Pic^0(X)$ [Lemma \ref{diagram 1}]. Since
		\begin{equation*}
			g^{-1}(\xi) =
			\begin{cases}
				Nm^{-1}_{\gammaeta}(\xi) &  \text{if}\ r \ \text{is odd},\\
				Nm^{-1}_{\gammaeta}(\xi\otimes \eta^{-r/2}) & \text{if}\ r\ \text{is even},
			\end{cases}
		\end{equation*} 
		it follows from \cite[Proposition 5.13]{GO18} that there exists a line bundle $\delta_\eta\in \Pic^0(X)[m]$ with the property that the subgroup $\langle \delta_\eta\rangle\simeq\mathbb{Z}/m\mathbb{Z}$ acts freely and transitively on $\pi_0(g^{-1}(\xi))$, while every element of $\Pic^0(X)[m]$ not in $ \langle \delta_\eta\rangle$ acts trivially on $\pi_0(g^{-1}(\xi))$.
		We can consider $\Pic^0(X)[m]$ as a subgroup of \,$\Gamma$. Thus, $\delta_\eta\in \Gamma$. This $\delta_\eta$ will be our candidate satisfying the statement.\\
		By Proposition \ref{inverse of connected under det is connected}\,, we know that the connected components of $\mathcal{N}^\textbf{t}_\eta$ are of the form 
		$$\{({{\det}_{\yeta}\,,\,tr_{\yeta})}^{-1}(C\times h^{-1}(0))\,\mid\,C\,\text{is a connected component of}\,g^{-1}(\xi)\}.$$
		To check the action is free, let 
		$$ \delta_\eta^e\cdot ({{\det}_{\yeta}\,,\,tr_{\yeta})}^{-1}(C\times h^{-1}(0))=({{\det}_{\yeta}\,,\,tr_{\yeta})}^{-1}(C\times h^{-1}(0))\,\,\text{ for some integer}\,\, e.$$
		For any $(F_*\,,\,\phi_*)\in ({{\det}_{\yeta}\,,\,tr_{\yeta})}^{-1}(C\times h^{-1}(0))$, this would imply that $\det(F\otimes {\gammaeta}^*\delta_\eta^e)\in C$\,,as well as $\det(F)\in C$\,. But
		\begin{align*}
			& \det(F\otimes {\gammaeta}^*(\delta_\eta^e))= \det F\otimes ({\gammaeta}^*(\delta_\eta^{l\cdot e})),\,\,\textnormal{which forces}\,\,\delta_\eta^{l\cdot e}\simeq \mathcal{O}_X\,,\\
			&\implies  m\mid le \,\,\,\,\,[\because \,\,\langle \delta_\eta\rangle\,\,\text{acts freely on}\,\,\pi_0(g^{-1}(\xi))]\\
			&\implies  m\mid e \,\,\,\,\,\,\,\,[\because\,\,gcd(l,m)=1 ]\\
			&\implies  \delta_\eta^e\simeq\mathcal{O}_{X}.
		\end{align*}
		To check transitivity, choose two connected components $C$ and $C'$ of $g^{-1}(\xi)$. Choose a parabolic Higgs bundle $(F_*\,,\,\phi_*)\in (\det_{\yeta}\,,\,tr_{\yeta})^{-1}(C\times h^{-1}(0))$. Since $\langle \delta_\eta\rangle$ acts transitively on $\pi_0(g^{-1}(\xi))$ and $\det(F)\in C$, there exists an integer $n$ such that 
		$$\det(F)\otimes {\gammaeta}^*(\delta_\eta^n)\in C'.$$
		Since $gcd(m,l)=1$, we can write $n=am+bl$ for some integers $a,b$.
		Since $ord(\delta_\eta)=m$, this implies 
		$$\det(F)\otimes {\gammaeta}^*(\delta_\eta^n)=\det(F)\otimes {\gammaeta}^*(\delta_\eta^{bl})=\det(F\otimes {\gammaeta}^*(\delta_\eta^{b}))\,\,\,\,[\because\,\,rank(F)=l],$$
		and thus $F_*\otimes {\gammaeta}^*(\delta_\eta^b)\in \det^{-1}(C')$\,. It follows that $\delta_\eta^b\cdot \det^{-1}(C\times h^{-1}(0)) = \det^{-1}(C'\times h^{-1}(0))$.
		
		Let $\mu\in \Gamma\setminus \langle \delta_\eta\rangle$. If $(F_*\,,\,\phi_*)\in (\det_{\yeta}\,,\,tr_{\yeta})^{-1}(C\times h^{-1}(0))$ for some component $C$ of $g^{-1}(\xi)$, then we claim that $\mu \cdot(F_*\,,\,\phi_*) = (F_*\otimes {\gammaeta}^*\mu\,,\,\phi_*\otimes Id_{\gammaeta^*\mu}) \in \det^{-1}(C\times h^{-1}(0))$ as well. This follows because
		$$\det(F\otimes {\gammaeta}^*\mu) = \det(F)\otimes {\gammaeta}^*(\mu^l)\,$$
		and $\mu$ is $r$-torsion implies that $\mu^l$ is $m$-torsion, since\, $lm=r$. Thus $\mu^l\in \Pic^0(X)[m]$. Since $gcd(l,m)=1$ and $\mu\notin \langle \delta_{\eta} \rangle$, clearly $\mu^l \notin \langle \delta_\eta\rangle$ as well. By \cite[Proposition 5.13]{GO18} we conclude that 
		$$\det(F)\otimes {\gammaeta}^*(\mu^l)\in C\,\,,$$
		and thus $F_*\otimes {\gammaeta}^*\mu \in \det^{-1}(C)$. It follows that $$\mu\cdot ({\det}_{\yeta}\,,\,tr_{\yeta})^{-1}(C\times h^{-1}(0)) = ({\det}_{\yeta}\,,\,tr_{\yeta})^{-1}(C\times h^{-1}(0))\,.$$
	\end{proof}
	Thus, whenever $r$ is a product of distinct primes, Proposition \ref{Gamma action on components} describes the $\Gamma$-action on the components of $\mathcal{N}^\textbf{t}_\eta$\,'s. Since the components of $\Metaalpha$ are given precisely by the components of $\mathcal{N}^\textbf{t}_{\eta}$\,'s for various $\textbf{t}\in  {\bf P}(\alpha)$ [Lemma \ref{components}], Proposition \ref{Gamma action on components}, in turn, gives us an understanding of how the action of $\Gamma$ permutes the components of $\MHetaalpha$\,.\\
	
	\section{Orbifold Euler characteristic of the quotient $\Malpha/\Gamma$}\label{orbifold euler characteristic section}
	
	Let $G$ be a finite group acting on a compact manifold $M$. Following 
	\cite{HH90}, the \textit{orbifold Euler characteristic} of $M/G$, denoted $\chi(M,G)$, can be defined as 
	\begin{align}
		\chi(M,G) := \sum_{[g]} \chi\left(M^g/C(g)\right)
	\end{align}
	where the sum is taken over a set of representatives for the conjugacy classes of $G$, $M^g$ denotes the fixed point set (which is a submanifold), $C(g)$ denotes the centralizer of $g$ and $\chi(M^g/C(g))$ is the usual topological Euler characteristic of the quotient space (Cf. \cite{HH90}).\\
	
	Henceforth, we drop the Higgs field (meaning we set the Higgs field to be zero), and work with stable parabolic bundles. Let $\Malpha$ denote the moduli of stable parabolic bundles of rank $r$, determinant $\xi$ and weights $\alpha$ of full-flag type.
	Since the group $\Gamma$ is abelian, the orbifold Euler characteristic for the orbifold $\Malpha/\Gamma$ takes the form
	\begin{align}\label{orbifold euler characteristic}
		\chi_{orb}(\Malpha\,,\,\Gamma) = \sum_{\eta\in \Gamma} \chi\left(\Metaalpha/\Gamma\right)\,.
	\end{align}
	Let $\eta\in \Gamma$ be non-trivial. For the summands occurring in the right-hand side of (\ref{orbifold euler characteristic}), note that as a vector space we always have
	\begin{align}\label{cohomology iso 1}
		H^*\left(\Metaalpha/\Gamma\,,\,\mathbb{C}\right) \simeq H^*\left(\Metaalpha\,,\,\mathbb{C}\right)^{\Gamma}\,.
	\end{align}	
	
	\subsection{The groups $H^*\left(\Metaalpha\,,\,\mathbb{C}\right)^{\Gamma}$}\label{equivariant cohomology subsection}
	
	Recall from the beginning of Section \ref{connected componets of fixed point loci} that ${\rm Gal}({\gammaeta})$ has an induced action on $ {\bf P}(\alpha)$. Consider the quotient map 
	$$ {\bf P}(\alpha)\longrightarrow  {\bf P}(\alpha)/{\rm Gal}({\gammaeta})\,,$$ 
	and fix a section of this quotient. Denote this section by $s$. For $\textbf{t}\in  {\bf P}(\alpha)$, let $[\textbf{t}]$ denote its class in $ {\bf P}(\alpha)/{\rm Gal}({\gammaeta})$.
	\begin{lemma}\label{equivariant cohomology lemma}
		Fix a choice of a section $s$ as above. For any non-trivial $\eta\in \Gamma$, we have the following isomorphism of cohomology groups:	
		\begin{align*}
			H^*\left(\Metaalpha\,,\, \mathbb{C}\right) \,\,\simeq \bigoplus_{[\textbf{t}]\in  {\bf P}(\alpha)/{\rm Gal}({\gammaeta})}  H^*\left(\mathcal{N}^{s([\textbf{t}])}_{\eta}\,,\,\mathbb{C}\right)\,.
		\end{align*}
	\end{lemma}
	\begin{proof}
		Since $\underset{\textbf{t}\in  {\bf P}(\alpha)}{\coprod} \mathcal{N}^\textbf{t}_\eta \,\,\xrightarrow{f} \Metaalpha $ is a principal ${\rm Gal}({\gammaeta})$-\,bundle [Corollary \ref{quotient}], we have 
		\begin{align}\label{iso3}
			H^k\left(\Metaalpha, \mathbb{C}\right) \simeq \left(\bigoplus_{\textbf{t}\in  {\bf P}(\alpha)}H^k\left(\mathcal{N}^\textbf{t}_{\eta} \,,\,\mathbb{C}\right)\right)^{{\rm Gal}({\gammaeta})}\,.
		\end{align}
		Let us denote 
		$$V := \bigoplus_{\textbf{t}\in  {\bf P}(\alpha)}H^k\left(\mathcal{N}^\textbf{t}_{\eta} \,,\,\mathbb{C}\right)\,\,,\,\, W_{[\bf t]} := \bigoplus_{i=0}^{m-1} H^k\left(\mathcal{N}^{\mu^is([\textbf{t}])}_{\eta}\,,\,\mathbb{C}\right)\,\,\forall [{\bf t}]\in  {\bf P}(\alpha)/{\rm Gal}(\gammaeta)\,.$$
		Clearly we have  $V= \underset{[{\bf t}]\in  {\bf P}(\alpha)/{\rm Gal}(\gammaeta)}{\bigoplus} W_{[{\bf t}]}$, and moreover each $W_{[{\bf t}]}$ is invariant under ${\rm Gal}(\gammaeta)$. It follows that 
		$$V^{{\rm Gal}(\gammaeta)} = \bigoplus_{[\textbf{t}]\in  {\bf P}(\alpha)/{\rm Gal}(\gammaeta)} W_{[{\bf t}]}^{{\rm Gal}(\gammaeta)}\,.$$
		Let $\mu$ be a generator of ${\rm Gal}({\gammaeta})$. For each $\textbf{t}\in  {\bf P}(\alpha)$, let us write the elements in the orbit of $s([\textbf{t}])$ in the following order: 
		$$\{s([\textbf{t}]),\mu s([\textbf{t}]),\mu^2s([\textbf{t}]),\cdots,\mu^{m-1}s([\textbf{t}])\}\,.$$
		As the action of $\mu$ on the cohomology group sends $H^k\left(\mathcal{N}^{\mu^{i+1}s([\textbf{t}])}_{\eta} \,,\,\mathbb{C}\right)$ to $H^k\left(\mathcal{N}^{\mu^i s([\textbf{t}])}_{\eta} \,,\,\mathbb{C}\right)$, we immediately see that a ${\rm Gal}({\gammaeta})$-invariant tuple from the summand $W_{[\bf t]}$
		must be of the form $$\left(\omega,(\mu^{m-1})^*\omega,\cdots,(\mu^2)^*\omega,\mu^*\omega\right)\,\,\textnormal{for some}\,\,\omega\in H^k\left(\mathcal{N}^{s([\textbf{t}])}_\eta\,,\mathbb{C}\right)\,.$$
		sending this tuple to $\omega$, we have
		$W_{[{\bf t}]}^{{\rm Gal}(\gammaeta)}\simeq H^k\left(\mathcal{N}^{s([\textbf{t}])}_\eta\,,\mathbb{C}\right)\,. $	It follows from (\ref{iso3}) that 
		$$H^k\left(\Metaalpha\,,\,\mathbb{C}\right) \simeq V^{{\rm Gal}(\gammaeta)} = \bigoplus_{[\textbf{t}]\in  {\bf P}(\alpha)/{\rm Gal}(\gammaeta)} W_{[{\bf t}]}^{{\rm Gal}(\gammaeta)} \simeq \bigoplus_{[\textbf{t}]\in  {\bf P}(\alpha)/{\rm Gal}(\gammaeta)} H^k\left(\mathcal{N}^{s([\textbf{t}])}_\eta\,,\mathbb{C}\right)\,. $$
	\end{proof}
	
	Next, let us prove a lemma that will aid us in computing cohomologies. As before, $l=r/m$ in what follows. As we saw in the the proof of Proposition \ref{inverse of connected under det is connected}, if we fix a line bundle $\mathcal{L}_{\eta}\in \Pic^d(\yeta)$, for any $\textbf{t}\in  {\bf P}(\alpha)$ we have the following fiber diagram:
	\begin{align}\label{diagram 2}
		\xymatrix{ \Pic^0(\yeta)\times M^\textbf{t}_{\yeta}(l,\mathcal{L}_{\eta})  \ar[d] \ar[r]^(0.63){\mu'_l} & M^\textbf{t}_{\yeta}(l,d) \ar[d]^{\det_{\yeta}} \\
			\Pic^0(\yeta) \ar[r]^{\mu_l} & \Pic^d(\yeta)
		}
	\end{align}
	where $\mu_l(L) := L^l\otimes \mathcal{L}_{\eta}$ \eqref{eqn:mu-l-map}, and $\mu'_l$ is its pull-back, given by 
	\begin{align}
		\mu'_l(L,F_*) = F_*\otimes L\,.
	\end{align} 
	The group $\Gamma$ acts on $\Pic^0(\yeta)$ and $M^t_{\yeta}(l,d)$, where an element $\delta\in\Gamma$ acts by sending $L\mapsto L\otimes \gammaeta^*(\delta)$ and $F_*\mapsto F_*\otimes \gammaeta^*(\delta)$ respectively. On the other hand, let us consider the $\Gamma$-\,action on $\Pic^d(\yeta)$ where an element $\delta\in \Gamma$ acts by sending 
	\begin{align}\label{eqn:gamma-action-1}
		L\mapsto L\otimes \gammaeta^*(\delta^l)\,.
	\end{align}
	Finally, we consider the $\Gamma$-\,action on $\Pic^0(\yeta)\times M^\textbf{t}_{\yeta}(l,\mathcal{L}_{\eta})$ by the same action as just described for the first component, and the trivial action for the second component. Namely, an element $\delta\in\Gamma$ acts by sending 
	\begin{align}\label{eqn:gamma-action-2}
		(L\,,\,F_*)\mapsto (L\otimes\gammaeta^*(\delta)\,,\,F_*).
	\end{align}
	It is easy to see that both $\mu_l$ and $\mu_l'$ in diagram (\ref{diagram 2}) are $\Gamma$-\,equivariant under these actions. 
	
	\begin{lemma}\label{lem2}
		Recall the map $g$ from Lemma \ref{lem1}. For each $\xi\in \Pic^d(X)$\,, $g^{-1}(\xi)$ is invariant under the $\Gamma$-\,action on $\Pic^d(\yeta)$ described above. As a consequence, both the restricted maps
		$$\mu_l^{-1}(g^{-1}(\xi))\overset{\mu_l}{\longrightarrow}g^{-1}(\xi) \,\,\,\,\text{and}\,\,\,\, \mu_l^{-1}(g^{-1}(\xi)) \times M^\textbf{t}_{\yeta}(l,\mathcal{L}_{\eta})\overset{\mu_l'}{\longrightarrow} \mathcal{N}^{\mathbf{t}}_{\eta}$$ are equivariant for the $\Gamma$-actions described above and in \S\ref{gamma action subsection} . 
	\end{lemma}
	\begin{proof}
		For any $L\in g^{-1}(\xi)$ and $\delta\in \Gamma$, we have
		\begin{align*}
			g(L\otimes\gammaeta^*(\delta^l)) &= g(L)\otimes Nm_{\gammaeta}(\gammaeta^*(\delta^l))\\
			&=g(L)\otimes \delta^{l\cdot m}\\
			&= g(L) \,\,\,[\because \,\,lm=r]\\
			&\simeq \xi\,\,,
		\end{align*}
		and thus $g^{-1}(\xi)$ is invariant under the action of $\Gamma$\,.
		The $\Gamma$-\,equivariance of the maps $\mu_l$ and $\mu_l'$ are obvious from their definitions.
	\end{proof}
	
	\begin{definition}[\text{\cite[Definition 3.4]{GO18}}]\label{prym definition} 
		We define the \textit{Prym variety} associated to the cover ${\gammaeta}: \yeta \rightarrow X$, denoted by $\textnormal{Prym}_{\gammaeta}(\yeta)$\,, as the connected component of $ker(Nm_{\gammaeta})$ containing $\mathcal{O}_{\yeta}$. It is an abelian subvariety of $\Pic^0(\yeta)$.
	\end{definition}
	
	\begin{proposition}\label{equivariant cohomology proposition}
		Fix a choice of a section $s$ as above. For any non-trivial $\eta\in \Gamma$, we have the following isomorphism of cohomology groups: fix a line bundle $\mathcal{L}_{\eta}\in \Pic^d(\yeta)$\,. Choose $s$ as in the beginning of \S \ref{equivariant cohomology subsection}. Then	
		\begin{align*}
			H^*\left(\Metaalpha/\Gamma, \mathbb{C}\right) \simeq \bigoplus_{[\textbf{t}]\in  {\bf P}(\alpha)/{\rm Gal}({\gammaeta})} H^*\left(\textnormal{Prym}_{\gammaeta}(Y_\eta)\,,\,\mathbb{C}\right) \otimes H^*\left(M^{s([\textbf{t}])}_{\yeta}(l,\mathcal{L}_{\eta})\,,\,\mathbb{C}\right)\,.
		\end{align*}
	\end{proposition}
	\begin{proof}
		We have already observed that (\ref{cohomology iso 1})
		\begin{align*}
			H^*\left(\Metaalpha/\Gamma\,,\,\mathbb{C}\right) \simeq H^*\left(\Metaalpha\,,\,\mathbb{C}\right)^{\Gamma}\,.
		\end{align*}
		As seen in \S \ref{gamma action subsection}, $\Gamma$ acts on each $\mathcal{N}^\textbf{t}_\eta$\,.	Under this action, the morphism $f$ from Corollary \ref{quotient} is $\Gamma$-equivariant. Thus, the isomorphism in Lemma \ref{equivariant cohomology lemma} is also $\Gamma$-equivariant, which implies that
		\begin{align}\label{iso2}
			H^*\left(\Metaalpha\,,\,\mathbb{C}\right)^\Gamma \simeq \bigoplus_{[\textbf{t}]\in  {\bf P}(\alpha)/{\rm Gal}({\gammaeta})}H^*\left(\mathcal{N}^{s([\textbf{t}])}_\eta\,,\,\mathbb{C}\right)^{\Gamma}\,.
		\end{align}
		Define	$\Gamma_l := \Pic^0(\yeta)[l]\simeq(\mathbb{Z}/l\mathbb{Z})^{2g_{\yeta}}\,.$
		The map $\mu_l$ makes $\Pic^0(\yeta)$ into a principal $\Gamma_l$-\,bundle over $\Pic^d(\yeta)$, since the multiplication-by-$l$ map of abelian varieties has this property. Thus its pull-back 
		$$\Pic^0(\yeta)\times M^\textbf{t}_{\yeta}(l,\mathcal{L}_{\eta}) \xrightarrow{\mu'_l} M^\textbf{t}_{\yeta}(l,d)$$
		is also a principal $\Gamma_l$-\,bundle, where the action on left-hand side is given diagonally as 
		\begin{align}\label{gamma_n action}
			{\gamma}\cdot(L\,,\,F_*)=(\,{\gamma}\otimes L\,,\, F_*\otimes {\gamma}^{\vee}\,) \,\,\forall\,{\gamma}\in \Gamma_l\,.
		\end{align}
		Restricting diagram (\ref{diagram 2}) to $\det_{\yeta}: \mathcal{N}^{s([t])}_\eta \rightarrow g^{-1}(\xi)$ would give rise to the fiber diagram
		\begin{align}\label{diagram 3}
			\xymatrix{ \mu_l^{-1}(g^{-1}(\xi))\times M^{s([\textbf{t}])}_{\yeta}(l,\mathcal{L}_{\eta}) \ar[r]^(0.7){\mu'_l} \ar[d] & \mathcal{N}^{s([\textbf{t}])}_\eta \ar[d]^{{\det}_{\yeta}} \\
				\mu_l^{-1}(g^{-1}(\xi)) \ar[r]^{\mu_l} & g^{-1}(\xi)
			}
		\end{align}
		where $\mu'_l$ is a principal $\Gamma_l$-\,bundle map. Thus
		\begin{align}
			H^*(\mathcal{N}^{s([\textbf{t}])}_\eta\,,\,\mathbb{C}) \simeq \left(H^*(\mu_l^{-1}(g^{-1}(\xi))\,,\,\mathbb{C})\otimes H^*(M^{s([\textbf{t}])}_{\yeta}(l,\mathcal{L}_{\eta})\,,\,\mathbb{C})\right)^{\Gamma_l}\,\,.
		\end{align}
		It is shown in \cite[Proposition 4.1]{BD10} that $\Gamma_l$ acts trivially on  $H^*(M^{s([\textbf{t}])}_{\yeta}(l,\mathcal{L}_{\eta}),\mathbb{C})$\,, which implies 
		\begin{align}\label{cohomology iso 2}
			H^*(\mathcal{N}^{s([\textbf{t}])}_\eta\,,\,\mathbb{C}) \simeq H^*\left(\mu_l^{-1}(g^{-1}(\xi))\,,\,\mathbb{C}\right)^{\Gamma_l}\otimes H^*(M^{s([\textbf{t}])}_{\yeta}(l,\mathcal{L}_{\eta})\,,\,\mathbb{C})\,\,.
		\end{align}
		Here we make a few remarks. First of all, although in \cite{BD10} the assumption of rank $r=2$, small weights and $\gcd(d,r)=1$ is taken throughout, the proof of \cite[Proposition 4.1]{BD10} only uses results from \cite{Ni86} and \cite{BR96} which are actually true for any rank and any system of generic weights. Thus the proof of \cite[Proposition 4.1]{BD10} holds in the more general case that we consider here. Also, their proof works with $\mathbb{C}$-coefficients as well.
		
		Now, the $\Gamma_l$-\,action on $\mu_l^{-1}(g^{-1}(\xi))\times M^{s([\textbf{t}])}_{\yeta}(l,\mathcal{L}_{\eta})$ which makes $\mu'_l$ a principal bundle map is given by (\ref{gamma_n action}). On the first component (namely $\mu_l^{-1}(g^{-1}(\xi))$), the $\Gamma_l$-\,action is precisely the same one which makes $\mu_l$ into a principal $\Gamma_l$-\,bundle map. Thus we have
		\begin{align}
			&H^*\left(\mu_l^{-1}(g^{-1}(\xi))\,,\,\mathbb{C}\right)^{\Gamma_l} \simeq H^*(g^{-1}(\xi)\,,\,\mathbb{C})\,\,.\label{cohomology iso 3}
		\end{align}
		Now, recall the $\Gamma$-\,actions described just before Lemma \ref{lem2}, namely \eqref{eqn:gamma-action-1} and \eqref{eqn:gamma-action-2}. These actions induce $\Gamma$-\,actions on their respective cohomology groups. By the same lemma, both $\mu_l$ and $\mu_l'$ are $\Gamma$-\,equivariant maps, and thus the map on cohomologies induced by $\mu_l$ and $\mu_l'$ are also $\Gamma$-\,equivariant. Since the isomorphism in (\ref{cohomology iso 2}) is precisely given by $\mu_l'^*$ on cohomologies, it must restrict to an isomorphism between its  $\Gamma$-\,invariant parts on both sides.\\	
		Recall that the $\Gamma$-\,action on $M^{s([\textbf{t}])}_{\yeta}(l,\mathcal{L}_{\eta})$ is taken to be trivial. Thus, taking $\Gamma$-\,invariant parts in both sides of the isomorphism (\ref{cohomology iso 2}) together with the isomorphism (\ref{cohomology iso 3}), we get
		\begin{align}\label{iso4}
			H^*(\mathcal{N}^{s([\textbf{t}])}_{\eta}\,,\,\mathbb{C})^{\Gamma} \simeq H^*(g^{-1}(\xi)\,,\,\mathbb{C})^\Gamma \otimes H^*(M^{s([\textbf{t}])}_{\yeta}(l,\mathcal{L}_{\eta})\,,\,\mathbb{C})\,.
		\end{align}
		Now, the group $\Pic^0(X)[m]$ also acts on $g^{-1}(\xi)$; an element $\tau\in \Pic^0(X)[m]$ acts by sending 
		$$L\mapsto L\otimes \gammaeta^*(\tau)\,\,,$$
		and it is shown in \cite[Proposition 6.6]{GO18} that under this action, the fixed point subgroup is 
		$$ H^*(g^{-1}(\xi)\,,\,\mathbb{C})^{\Pic^0(X)[m]} = H^*(\textnormal{Prym}_{\gammaeta}(\yeta)\,,\,\mathbb{C})\,\,.$$
		From the surjectivity of the multiplication-by-$l$ map on the abelian variety $\Pic^0(X)$, it easily follows that the morphism 
		\begin{align}
			\Gamma & \longrightarrow\Pic^0(X)[m] \\
			\delta & \mapsto \delta^l
		\end{align}
		is surjective. Thus it follows that $\Gamma$ acts on $g^{-1}(\xi)$ the same way as $\Pic^0(X)[m]$, from which we conclude that
		$$ H^*(g^{-1}(\xi)\,,\,\mathbb{C})^{\Gamma} = H^*(\textnormal{Prym}_{\gammaeta}(\yeta)\,,\,\mathbb{C})\,.$$
		This, together with the isomorphisms (\ref{iso2}) and (\ref{iso4}), proves our claim.
	\end{proof}
	\begin{theorem}\label{orbifold euler characteristic corollary}
		Fix a natural number $r$. Let $X$ be a smooth connected complex projective curve of genus $g\geq 2$ (if $g=2$, we assume $r\geq 3$). Let $\alpha$ be a system of generic weights of full-flag type \eqref{def:generic-weight}. Let $\Malpha$ denote the moduli of stable parabolic bundles of rank $r$, determinant $\xi$ and weights $\alpha$. Let $\Gamma$ be the subgroup of $r$-torsion points in $\Pic^0(X)$. The orbifold Euler characteristic \textnormal{[}\eqref{orbifold-euler-char-definition},\eqref{orbifold euler characteristic}\textnormal{]} of   $\Malpha/\Gamma$ is given by
		$$\chi_{orb}\left(\Malpha\,,\,\Gamma\right) = \chi\left(\Malpha\right)\,,$$
		where the right-hand side denotes the usual Euler characteristic.
	\end{theorem}
	
	\begin{proof}
		The isomorphism in Proposition \ref{equivariant cohomology proposition} preserves the grading on both sides, where the right-hand side is the tensor product of graded vector spaces. Thus, for each $k\geq 0$ we have
		\begin{align*}
			\dim H^k(\Metaalpha/\Gamma) &= \sum_{[{\bf t}]\in  {\bf P}(\alpha)/{\rm Gal}(\gammaeta)} \,\sum_{i+j=k} \dim H^i\left(\textnormal{Prym}_{\gammaeta}(\yeta)\right)\cdot\dim H^j\left(M^{s([{\bf t}])}_{\yeta}(l\,,\,\mathcal{L}_\eta)\right)\\
			\implies \chi\left(\Metaalpha/\Gamma\right) &= \sum_{k\geq 0}(-1)^k\dim H^k\left(\Metaalpha/\Gamma\right) \\
			&= \sum_{[{\bf t}]\in  {\bf P}(\alpha)/{\rm Gal}(\gammaeta)} \,\sum_{k\geq 0}\sum_{i+j=k} (-1)^i\dim H^i\left(\textnormal{Prym}_{\gammaeta}(\yeta)\right)\cdot(-1)^j\dim H^j\left(M^{s([{\bf t}])}_{\yeta}(l\,,\,\mathcal{L}_\eta)\right)\\
			&= \sum_{[{\bf t}]\in  {\bf P}(\alpha)/{\rm Gal}(\gammaeta)} \chi\left(\textnormal{Prym}_{\gammaeta}(\yeta)\right)\cdot \chi\left(M^{s([{\bf t}])}_{\yeta}(l\,,\,\mathcal{L}_\eta)\right)\,.
		\end{align*}
		For non-trivial $\eta\in\Gamma$, $\textnormal{Prym}_{\gammaeta}(\yeta)$ is a complex torus, and thus $\chi\left(\textnormal{Prym}_{\gammaeta}(\yeta)\right) = 0$ whenever $\eta$ is non-trivial. On the other hand, when $\eta = \mathcal{O}_{\yeta}$, we have $m=1$, $\yeta =X$ and $\gammaeta = \textnormal{Id}_X$\,, and we can choose $\mathcal{L}_{\eta}= \xi$ in that case.	From (\ref{orbifold euler characteristic}) we see immediately that
		\begin{align}
			\chi_{orb}\left(\Metaalpha/\Gamma\right) = \chi (\Malpha)\,,
		\end{align}
		as claimed.
	\end{proof}
	
	\section{Chen--Ruan cohomology of the orbifold $\Malpha/\Gamma$}\label{chen-ruan cohomology section}
	To avoid notational  cumbersomeness we shall restrict ourselves in defining the Chen--Ruan cohomology groups in the special case when the orbifold is a global quotient $Y/G$ for a compact complex manifold $Y$ under the action a finite \textit{abelian} group $G$, since this will be our case of interest. We refer to \cite[\S 2]{FG03} for the general definition of Chen--Ruan cohomology groups for global quotient orbifolds (see \cite{CR04} for a more general definition).\\
	So, let $Y$ be a compact complex manifold with an action of a finite abelian group $G$. The Chen--Ruan cohomology group of $Y/G$ is denoted by $H^*_{CR}(Y/G)$, where the
	grading is by rational numbers. To define the rational grading, we need the notion of degree-shift numbers.
	\begin{definition}\label{degree-shift number definition}
		Let $Y$ and $G$ be as above, and let $d$ be the dimension of $Y$ as a complex manifold. For $g\in G$ and $y\in Y^g$, consider the induced linear action of $g$ on the tangent space $T_{y}Y$. Since $g$ is of finite order, the eigenvalues are all roots of unity. Let $\lambda_1,\cdots,\lambda_d$ be the eigenvalues. Write $\lambda_j = \exp(2\pi\sqrt{-1}w_j)$, where $0\leq w_j<1$ are rational numbers. The \textit{degree-shift number} of $g$ at the point $y$ is defined to be the number 
		\begin{align}\label{degree-shift number general definition}
			\iota(g,y):=\sum_{j=1}^{d}w_j\,.
		\end{align}
	\end{definition}
	This is clearly a non-negative rational number in general. $\iota(g,y)$ only depends on the connected component $Z$ of $Y^g$ containing $y$, so it can also be denoted as $\iota(g,Z) $.
	
	The vector space $H^*_{CR}(Y/G)$ is given a rational grading as follows. Let $g\in G$, and let $Z$ be a connected component of $Y^g$. For each non-negative integer $i$, we assign the degree $i+ 2\iota(g,Z)$ to the elements in the summand $H^i(Z)$ of $H^i(Y^g)$.\\
	There is also a product structure on the Chen--Ruan cohomology which makes it into a graded ring; this product shall be described in \S\ref{Chen-Ruan description for moduli}\,in our case of consideration.\\
	Let us come back to our situation. Let $\eta\in \Gamma$ be non-trivial. Recall that $\Malpha$ is a compact complex manifold on which the group $\Gamma$ acts through tensorization. Consider the automorphism
	\begin{align}\label{map 1}
		\phi_\eta : \Malpha &\rightarrow \Malpha\\ 
		E_*&\mapsto E_*\otimes \eta\,\,.
	\end{align}
	If $E_*\in \Metaalpha$, the differential $d\phi_{\eta}(E_*) : T_{E_*}(\Malpha)\longrightarrow T_{E_*}(\Malpha)\,\,$ is a linear automorphism of vector space. Let $Z$ be the component of $\Metaalpha$ containing $E_*$. The \textit{degree-shift number} of $\eta$ at $Z$, denoted by $\iota(\eta\,,Z)$, is defined exactly as in (\ref{degree-shift number general definition}).
	
	\subsection{Degree-shift number computation}\label{degree-shift number computation}
	We shall now compute the degree-shift numbers under the following additional assumptions on the rank and degree:
	\begin{enumerate}[$(a)$]
		\item the rank $r$ is a product of distinct primes, \label{condition a}
		\item $r$ and $d$ are coprime. \label{condition b}
	\end{enumerate}
	Recall the definition of $\mathcal{N}^{\bf t}_{\eta}$ given just before Lemma \ref{parabolic fixed point}.
	\begin{proposition}\label{underlying stable}
		Let $\eta\in \Gamma\setminus \{\mathcal{O}_X\}$. Under assumption (\ref{condition b}) above, for each ${\bf t}\in  {\bf P}(\alpha)$, there exists a nonempty open subset $U_{\bf t}$ of $\mathcal{N}^{\bf t}_{\eta}$ with the following properties: for each parabolic bundle $F_*\in U_{\bf t}$\,,
		\begin{enumerate}[$(i)$]
			\item  the underlying bundle $F$ is stable, and 
			\item if $E_* = \gammaeta_*(F_*)$, then the underlying bundle $E$ is also stable.
		\end{enumerate} 
	\end{proposition}
	\begin{proof}
		Let $m=ord(\eta)$, and let $\gammaeta : \yeta \rightarrow X$ be a spectral curve corresponding to $\eta$. Let us consider parabolic bundles on $\yeta$ of rank $l=r/m$ having full-flag quasi-parabolic structures at the parabolic points $\gammaeta^{-1}(S)$ and a system of weights $\beta$ which is sufficiently small in the sense of \cite[Proposition 5.3]{BY99}. By the same proposition [\textit{loc. cit.}], we know that for any such parabolic stable bundle $F_*$, the underlying bundle $F$ is semistable as an usual bundle.\\
		Let $M_{\yeta}^{\beta}(l,d)$ denote the corresponding moduli of stable parabolic bundles over $\yeta$. Using \cite[Lemma 3.6]{BH10}, we know that there exists a birational map between $M^{\beta}_{\yeta}(l,d)$ and $M^{\bf t}_{\yeta}(l,d)$. The birational map simply replaces the weights, leaving the underlying bundle unchanged. Thus we get an open subset $V_{\bf t}\subset M^{\bf t}_{\yeta}(l,d)$ with the property that the parabolic bundles lying in $V_{\bf t}$ have underlying bundle semistable.\\
		Let $F_*\in V_{\bf t}$ with $\det(\gammaeta_*(F))\simeq \xi'$. Let $L\in \Pic^0(X)$ be such that $L^r\simeq \xi\otimes \xi'^{-1}$. Then 
		$$\det(\gammaeta_*(F\otimes \gammaeta^*L)) \simeq \det(\gammaeta_*(F)\otimes L) \simeq \det(\gammaeta_*(F))\otimes L^r \simeq \xi\,,$$
		where the first isomorphism is by projection formula. Clearly, tensoring by $\gammaeta^*L$ defines an automorphism $\phi_L$ of $M^{\bf t}_{\yeta}(l,d)$, and the above argument shows that 
		$$\phi_L(V_{\bf t})\cap \mathcal{N}^{\bf t}_{\eta}\neq \emptyset\,.$$
		
		Thus, if we define 
		$$U_{\bf t} := \phi_L(V_{\bf t})\cap \mathcal{N}^{\bf t}_{\eta}\,,$$
		then the parabolic bundles lying in $U_{\bf t}$ have the property that their underlying bundles are semistable. In fact, notice that condition $(\ref{condition b})$, namely $gcd(r,d)= 1$, also implies that 
		$gcd(l,d) = 1.$ Thus, the parabolic bundles lying in $U_{\bf t}$ have their underlying bundles stable. Since $gcd(r,d)=1$ it follows that the underlying bundle $E$ of $E_* = \gammaeta_*(F_*)$ is stable as well.
	\end{proof}
	
	For any ${\bf t}\in  {\bf P}(\alpha)$, choose an open subset $U_{\bf t}$ as above. In Lemma \ref{components} we have seen that  the map $f^{\bf t}: \mathcal{N}^{\bf t}_{\eta}\rightarrow \Metaalpha$ is an open embedding. Thus, $f^{\bf t}(U_{\bf t})$ is a nonempty open subset of $f^{\bf t}(\mathcal{N}^{\bf t}_{\eta})$. Moreover, By Proposition \ref{inverse of connected under det is connected}, $f^{\bf t}(\mathcal{N}^{\bf t}_\eta)$ has $m$ connected components, and moreover these components are permuted among themselves under $\Gamma$-\,action due to our assumption $(\ref{condition a})$ and Proposition \ref{Gamma action on components} .  \\
	From this observation, we can conclude the following.
	
	\begin{proposition}\label{degree-shift number remains same}
		Under assumption (\ref{condition a}), the degree-shift number remains unchanged for those connected components contained in $f^{\bf t}(\mathcal{N}^{\bf t}_{\eta})$\,.
	\end{proposition}
	\begin{proof}
		For any $\delta\in \Gamma$, if $$\phi_{\delta}:\Malpha\rightarrow \Malpha$$ denote the automorphism induced by tensoring with $\delta$, then we have the following commutative diagram:
		\begin{align}
			\xymatrix{
				\Malpha \ar[r]^{\phi_{\eta}} \ar[d]_{\phi_{\delta}} & \Malpha \ar[d]^{\phi_{\delta}}\\
				\Malpha \ar[r]^{\phi_{\eta}} & \Malpha
			}
		\end{align}
		which, in turn, induces the following commutative diagram  for any $E_*\in \Metaalpha$:
		
		\begin{align}
			\xymatrix@=2cm{
				T_{E_*}(\Malpha) \ar[r]^{d\phi_{\eta}(E_*)} \ar[d]_{d\phi_{\delta}(E_*)} & T_{E_*}(\Malpha) \ar[d]^{d\phi_{\delta}(E_*)}\\
				T_{E_*\otimes \delta}(\Malpha) \ar[r]^{d\phi_{\eta}(E_*\otimes \delta)} & T_{E_*\otimes\delta}(\Malpha)
			}
		\end{align}
		Where all the arrows are isomorphisms. Clearly $E_*\otimes \delta\in \Malpha$ as well. From this diagram it follows that the eigenvalues and their multiplicities for both horizontal maps above are the same, and hence the degree-shift number for $\eta$ computed at $E_*$ and $E_*\otimes \delta$ remains the same. From this and the fact that the $\Gamma$--action permutes the connected components of $f^{\bf t}(\mathcal{N}^{\bf t}_{\eta})$ transitively [Proposition \ref{Gamma action on components}], the proposition follows.
	\end{proof}
	
	\subsection{Action on the tangent bundle} 
	Proposition \ref{degree-shift number remains same} says that to compute the degree-shift number of a certain $\eta\in \Gamma\setminus \{\mathcal{O}_X\}$, it is enough to consider the degree-shift number for a single $E_*$ lying in $f^{\bf t}(\mathcal{N}^{\bf t}_{\eta})$ for each ${\bf t}\in  {\bf P}(\alpha)$. Using Proposition \ref{underlying stable}, we can choose an $E_*\in f^{\bf t}(\mathcal{N}^{\bf t}_{\eta})$ which further satisfies that the underlying bundle $E$ is stable, and moreover $E_* = \gammaeta_*(F_*)$ for some parabolic stable bundle $F_*$ whose underlying  bundle $F$ is stable as well.\\
	Following \cite[\S 3]{BP10}, we get
	\begin{align}
		\gammaeta^*(\mathcal{E}nd(E)) = \bigoplus_{u\in {\rm Gal}(\gammaeta)}\bigoplus_{v\in {\rm Gal}(\gammaeta)} \mathcal{H}om(u^*F, (vu)^*F)\,,
	\end{align}
	and for each $v\in {\rm Gal}(\gammaeta)$\,, the vector bundle
	\begin{align}\label{def 1}
		\mathcal{F}_v := \bigoplus_{u\in {\rm Gal}(\gammaeta)} \mathcal{H}om\left(u^*F,(vu)^*F\right)
	\end{align}
	is left invariant by the natural ${\rm Gal}(\gammaeta)$-\,action, and thus descends to a vector bundle on $X$. If $$\mathcal{E}_v\rightarrow X$$
	denotes the descent, then 
	\begin{align}\label{descentbundle 1}
		\mathcal{E}_v \,=\, \gammaeta_*\left(\mathcal{H}om(F,v^*F)\right)\,,\,\,\,\text{and}\,\,\,\gammaeta^*(\mathcal{E}_v) = \mathcal{F}_v\,.
	\end{align}	 
	An exactly similar reasoning can be applied in the parabolic setting. In the proof of Lemma \ref{parabolic fixed point} we saw that  
	$$\gammaeta^*(E_*) \, = \,\bigoplus_{u\in {\rm Gal}(\gammaeta)}u^*(F_*)\,,$$
	which gives the following decomposition:
	\begin{align}
		\gammaeta^*\left(\mathcal{P}ar\mathcal{E}nd(E_*)\right) = \bigoplus_{u\in {\rm Gal}(\gammaeta)}\bigoplus_{v\in
			{\rm Gal}(\gammaeta)}\mathcal{P}ar\mathcal{H}om\left(u^*F_*\,,\,(vu)^*F_*\right)  
	\end{align}
	For each $v\,\in\, {\rm Gal}(\gammaeta)$, the parabolic bundle
	\begin{align}\label{def 2}
		\mathcal{F}_{v*} \,:=\, \bigoplus_{u\in {\rm Gal}(\gammaeta)}\mathcal{P}ar\mathcal{H}om\left(u^*F_*\,,\,(vu)^*F_*\right)
	\end{align}
	is ${\rm Gal}(\gammaeta)$-\,equivariant, and thus descends to $X$. If 
	$$\mathcal{E}_{v*}\rightarrow X$$
	denotes the descent, then as in (\ref{descentbundle 1}),
	\begin{align}\label{descentbundle 2}
		\mathcal{E}_{v*} = \gammaeta_*\left(\mathcal{P}ar\mathcal{H}om(F_*,v^*F_*)\right)\,,\,\,\,\text{and}\,\,\,\gammaeta^*(\mathcal{E}_{v*}) = \mathcal{F}_{v*}\,.
	\end{align}
	From the decomposition 
	\begin{align}
		\gammaeta^*\left(\mathcal{P}ar\mathcal{E}nd(E_*)\right) = \bigoplus_{v\in {\rm Gal}(\gammaeta)}\mathcal{F}_{v*}
	\end{align} 
	and by uniqueness of descent, it also follows that
	\begin{align}\label{decomposition 1}
		\mathcal{P}ar\mathcal{E}nd(E_*) = \bigoplus_{v\in {\rm Gal}(\gammaeta)}\mathcal{E}_{v*}\,.
	\end{align}
	
	We will now describe the differential $d\phi_{\eta}(E_*)$ for an $E_*\in \Metaalpha$, where $\phi_{\eta}$ is the automorphism defined in (\ref{map 1}). Choose a parabolic isomorphism $\psi_*: E_*\rightarrow E_*\otimes \eta$. This induces a parabolic isomorphism on the endomorphism bundle:
	
	\begin{align*}
		\widehat{\psi_*}: \mathcal{P}ar\mathcal{E}nd(E_*) \longrightarrow \mathcal{P}ar\mathcal{E}nd(E_*\otimes \eta)  \simeq (E_*\otimes \eta)^*\otimes (E_*\otimes \eta)
		& \simeq (E_*\otimes E)\\
		& = \mathcal{P}ar\mathcal{E}nd(E_*)
	\end{align*}
	Since two such isomorphisms $\psi_*$ and $\psi'_*$ differ by a constant scalar, the induced map $\widehat{\psi_*}$ on the endomorphism bundle is independent of the choice of $\psi_*$.\\
	Let $$\mathcal{P}ar\mathcal{E}nd_0(E_*)\subset \mathcal{P}ar\mathcal{E}nd(E_*)$$
	denote the subsheaf of trace zero parabolic endomorphisms. Clearly $$\widehat{\psi_*}\left(\mathcal{P}ar\mathcal{E}nd_0(E_*)\right)\subset \mathcal{P}ar\mathcal{E}nd_0(E_*)\,.$$
	Let \begin{align}\label{tangent space map}
		\overline{\psi_*}: H^1\left(X,\mathcal{P}ar\mathcal{E}nd_0(E_*)\right)\longrightarrow H^1\left(X,\mathcal{P}ar\mathcal{E}nd_0(E_*)\right)
	\end{align}
	denote the map induced on the cohomology by $\widehat{\psi_*}$. From the construction it can be seen that the differential
	$$d\phi_{\eta} : T_{E_*}(\Malpha) \longrightarrow T_{E_*}(\Malpha)$$
	coincides with $\overline{\psi_*}$ in (\ref{tangent space map}) (see also \cite[\S 3]{BP10}).
	
	\begin{lemma}\label{eigenspaces of parend}
		Let $v\in {\rm Gal}(\gammaeta)\setminus\{1\}$. Consider the bundle $\mathcal{E}_{v*}$ from (\ref{descentbundle 2}). We have 
		$$H^0(X,\mathcal{E}_{v*}) = 0\,\, \text{and}\,\, \mathcal{E}_{v*}\subset \mathcal{P}ar\mathcal{E}nd_0(E_*). $$
	\end{lemma}
	\begin{proof}
		We have 
		$$\mathcal{P}ar\mathcal{H}om(F_*,v^*F_*)\subset \mathcal{H}om(F,v^*F)$$
		as a subsheaf. Since $\gammaeta$ is a finite map, it follows that the pushforward map $\gammaeta_*$ is exact. Thus 
		$$\mathcal{E}_{v*}\subset \mathcal{E}_v\,\,\,\,\,\,[\text{from (\ref{descentbundle 1}) and (\ref{descentbundle 2})}]$$
		as a subsheaf. Now both our claims follow from \cite[Lemma 3.1]{BP10}, which tell us that $H^0(X,\mathcal{E}_v) = 0$ and the trace map restricted to $\mathcal{E}_v$ is zero.
	\end{proof}
	
	We know that ${\rm Gal}(\gammaeta)\simeq \mu_m\subset \mathbb{C}^*$, the group of $m$-th roots of unity. In the proof of Lemma \ref{parabolic fixed point} we have seen that 
	\begin{align}
		\gammaeta^*(E_*) \simeq \bigoplus_{\sigma\in {\rm Gal}(\gammaeta)}\sigma^*(F_*)
	\end{align}
	decomposes $\gammaeta^*(E_*)$ into eigenspace sub-bundles under a suitably chosen endomorphism of $\gammaeta^*(E_*)$ using $\gammaeta^*(\psi_*)$ and the tautological trivialization of $\gammaeta^*(\eta)$ (see Lemma \ref{parabolic fixed point} for details). Under this decomposition, $v^*(F_*)$ is the $v$-\,eigenspace sub-bundle. From this description it easily follows that $\widehat{\psi_*}$ acts on $\mathcal{P}ar\mathcal{H}om(F_*,v^*F_*)$ as multiplication by $v$. Thus $\overline{\psi_*}$ acts as multiplication by $v$ on $$H^1\left(\yeta,\mathcal{P}ar\mathcal{H}om(F_*,v^*F_*)\right) = H^1\left(X,\gammaeta_*(\mathcal{P}ar\mathcal{H}om(F_*,v^*F_*))\right) \underset{(\ref{descentbundle 2})}{=} H^1(X,\mathcal{E}_{v*})\,.$$
	(The first equality follows because $\gammaeta$ is a finite morphism.)\\
	On the other hand, the direct sum decomposition (\ref{decomposition 1}) and Lemma \ref{eigenspaces of parend} implies that
	\begin{align}
		\mathcal{P}ar\mathcal{E}nd_0(E_*) = (\mathcal{P}ar\mathcal{E}nd_0(E_*)\cap \mathcal{E}_{1*})\bigoplus\bigoplus_{v\neq 1} \mathcal{E}_{v*}
	\end{align}
	This gives a direct sum decomposition of $H^1\left(X,\mathcal{P}ar\mathcal{E}nd_0(E_*)\right)$ into $v$-\,eigenspaces for $\overline{\psi_*}$.\\
	It remains to compute the dimension of these eigenspaces. We first compute it for each $v\neq 1$. Recall that we had fixed ${\bf t}\in  {\bf P}(\alpha)$. Consider the short exact sequence of sheaves on $\yeta$:
	\begin{align}
		0\longrightarrow \mathcal{P}ar\mathcal{H}om(F_*,v^*F_*)\longrightarrow \mathcal{H}om(F,v^*F)\longrightarrow \mathcal{K}_{v,\bf t}\longrightarrow 0
	\end{align}
	where $\mathcal{K}_{v,\bf t}$ is a skyscraper sheaf supported at the parabolic points $\gammaeta^{-1}(S)$. Consider the corresponding long exact sequence in cohomology. Using the facts
	$$H^i\left(\yeta,\mathcal{P}ar\mathcal{H}om(F_*,v^*F_*)\right) = H^i(X,\mathcal{E}_{v*})\,,\, H^i\left(\yeta,\mathcal{H}om(F,v^*F)\right) = H^i(X,\mathcal{E}_v)\,,\, H^0(X,\mathcal{E}_v)=0\,\,\,\text{\cite[Lemma 3.1]{BP10}}$$
	we get the following exact sequence:
	\begin{align}\label{long exact sequence}
		0\longrightarrow H^0(\yeta, \mathcal{K}_{v,\bf t})\longrightarrow H^1(X,\mathcal{E}_{v*})\longrightarrow H^1(X,\mathcal{E}_v)\longrightarrow 0\,.
	\end{align}
	
	The number $\dim H^0(\mathcal{K}_{v,\bf t})$ can be obtained from the proof of \cite[Lemma 2.4]{BH10}; let us translate their result in our situation. In the beginning of \S\ref{estimate} we described how the partition $\bf t\in  {\bf P}(\alpha)$ induces a parabolic structure on the points of $\gammaeta^{-1}(S)$ from the given system of weights $\alpha$.  Let $\alpha({\bf t})$ denote the system of weights obtained on the points of $\gammaeta^{-1}(S)$\, in this way; thus for a bundle $F$ on $\yeta$, for each $q\in \gammaeta^{-1}(S)$ the parabolic structure at the fiber $F_q$ is given by a full-flag filtration of $F_q$ and the chain of weights
	\begin{align*}
		\alpha({\bf t})^q_1 < \alpha({\bf t})^q_2<\cdots<\alpha({\bf t})^q_l\,\,\,(\text{where}\,\,l=r/m).
	\end{align*}
	As the parabolic structure at $(v^*F)_q$ is same as the parabolic structure at $F_{vq}$, it follows from \cite[Lemma 3.2]{BH10} that
	\begin{align*}
		\dim H^0(\mathcal{K}_{v,\bf t}) = \sum_{q\in \gammaeta^{-1}(S)}\#\{(j,k)\mid \alpha({\bf t})^q_j>\alpha({\bf t})^{vq}_k\}
	\end{align*}
	It is known that 
	$$\dim H^1(X,\mathcal{E}_v) = r^2(g-1)/m\,.\,\,\,\,\,\,\,\,\,\,\,\,\,\,\text{\cite[Lemma 3.2]{BP10}}$$
	
	Thus from (\ref{long exact sequence}) we get 
	\begin{align}
		\dim H^1(X,\mathcal{E}_{v*}) = \frac{r^2}{m}(g-1)\,+\, \sum_{q\in \gammaeta^{-1}(S)}\#\{(j,k)\mid \alpha({\bf t})^q_j>\alpha({\bf t})^{vq}_k\}\,.
	\end{align}
	We have thus proved the following result.
	\begin{proposition}\label{degree-shift number proposition}
		For ${\bf t}\in  {\bf P}(\alpha)$, let $E_*\in f(\mathcal{N}^{\bf t}_{\eta})\subset \Metaalpha$, where $f$ is as in Lemma \ref{parabolic fixed point}. The eigenvalues of the differential 
		$$d\phi_{\eta}(E_*) : T_{E_*}(\Malpha)\longrightarrow T_{E_*}(\Malpha)$$
		are $\mu_m$. For any $v\neq 1$, the multiplicity of the eigenvalue $v$ is given by 
		$$ \frac{r^2}{m}(g-1)\,+\, \sum_{q\in \gammaeta^{-1}(S)}\#\{(j,k)\mid \alpha({\bf t})^q_j>\alpha({\bf t})^{vq}_k\} \,,$$
		where the system of weights $\alpha({\bf t})$ on the points of $\gammaeta^{-1}(S)$ is described as above. 
	\end{proposition}
	
	\begin{corollary}\label{degree-shift number collary}
		For each non-trivial $\eta\in \Gamma$ with $ord(\eta) = m$, the degree-shift numbers are given as follows:  fix any ${\bf t}\in  {\bf P}(\alpha)$. If $Z$ is a component of $\Metaalpha$ contained in $f(\mathcal{N}^{\bf t}_{\eta})$\,, the degree-shift number of $\eta$ for the component $Z$ is given by 
		$$\iota(\eta, Z) = \sum_{i=1}^{m-1} \left[\dfrac{i\cdot r^2}{m^2}(g-1)\,+\, \frac{i}{m}\sum_{q\in \gammaeta^{-1}(S)}\#\Bigg\{(j,k)\mid \alpha({\bf t})^q_j> \alpha({\bf t})^{\mu_i\cdot q}_k\Bigg\}\right]\,,$$
		where $\mu_i := e^{i\cdot2\pi\sqrt{-1}/m}\,.$   
	\end{corollary}
	
	\begin{remark}
		Let us consider the case when rank  $r=2$. Using the same notations as above, we have that any non-trivial $\eta\in \Gamma$ has order $2$, and $F_*$ is a parabolic line bundle. In this case, ${\rm Gal}(\gammaeta) \simeq \mu_2 = \{1,-1\}$ with $-1 = e^{\pi\sqrt{-1}}$. It is not hard to see that for each $q\in \gammaeta^{-1}(S)$,
		$$\#\{(j,k)\mid \alpha({\bf t})^q_j>\alpha({\bf t})^{\xi q}_k\} = 0 \,\,\text{or}\,\,1$$
		depending on the partition $\bf t$. From this, it easily follows that 
		$$\sum_{q\in \gammaeta^{-1}(S)}\#\{(i,j)\mid \alpha({\bf t})^q_i>\alpha({\bf t})^{\xi q}_j\} = |S|\,,$$
		and Proposition \ref{degree-shift number proposition} implies that the multiplicity of the eigenvalue $\xi$ is given by 
		$$2(g-1) \,+\, |S|\,.$$
		This agrees with the multiplicity computed in \cite[Lemma 3.1]{BD10}. Consequently, for any component of $\Metaalpha$, the degree-shift number equals $$(g-1)\,+\,|S|/2\,,$$
		which recovers the result of \cite[Corollary 3.2]{BD10}.
	\end{remark}	
	
	\subsection{Chen--Ruan cohomology of $\Malpha/\Gamma$}\label{Chen-Ruan description for moduli}
	We shall stick to our assumptions on rank and degree as mentioned in \S\ref{degree-shift number computation}.
	We note that the natural map $$H^*\left(\Malpha/\Gamma\,,\,\mathbb{C}\right)\rightarrow H^*\left(\Malpha\,,\,\mathbb{C}\right)$$
	induced from the projection $\Malpha \rightarrow \Malpha/\Gamma$ is an isomorphism of graded rings \cite[Proposition 4.1]{BD10}.

	As discussed before, this vector space is graded by non-negative rational numbers. We describe its graded pieces below. \\
	For ease of notation let us denote $G\,\,=\,\,{\rm Gal}(\gammaeta)$.
	Following same conventions as in \S\ref{equivariant cohomology subsection}, let us fix a section $s$ of the quotient map $ {\bf P}(\alpha)\rightarrow  {\bf P}(\alpha)/G$\,. Due to Lemma \ref{components}, for each non-trivial $\eta\in \Gamma$,
	$$\Metaalpha = \coprod_{[{\bf t}]\in  {\bf P}(\alpha)/G} f(\mathcal{N}^{s([{\bf t}])}_{\eta})\,.$$
	By Lemma \ref{Gamma action on components}, it follows that for each $s([{\bf t}])\in  {\bf P}(\alpha)$, the connected components contained in $f(\mathcal{N}^{s([{\bf t}])}_{\eta})$ get permuted transitively among themselves under the $\Gamma$--action on $\Metaalpha$\,. Thus, if we denote 
	\begin{align}\label{component description}
		Z^{[{\bf t}]}_{\eta} := f(\mathcal{N}^{s([{\bf t}])}_{\eta})/\Gamma\,,
	\end{align}
	then $Z^{[{\bf t}]}_{\eta}$\,'s are precisely the connected components of $\Metaalpha/\Gamma$\,, and we obtain 
	\begin{align}\label{components 2}
		\Metaalpha/\Gamma = \coprod_{[{\bf t}]\in  {\bf P}(\alpha)/G}Z^{[{\bf t}]}_{\eta}
	\end{align}   
	as the decomposition of $\Metaalpha/\Gamma$ into connected components.\\
	Moreover, in view of Corollary \ref{degree-shift number collary}, we can define 
	$$\iota\left(\eta\,,\,Z^{[{\bf t}]}_{\eta}\right)$$
	to be the number $\iota(\eta,Z)$ for any one of the components $Z$ contained in $f(\mathcal{N}^{s([{\bf t}])}_{\eta})$\,. 
	
	\begin{definition}\label{chen-ruan definition}
		For each rational number $i$, the $i$-th \textit{Chen--Ruan cohomology group} of the orbifold $\Malpha/\Gamma$ is defined as
		\begin{align}
			H^i_{CR}\left(\Malpha/\Gamma\right) =  H^i\left(\Malpha\,,\,\mathbb{C}\right)\bigoplus\,\bigoplus_{\eta\in\Gamma\setminus\{\mathcal{O}_X\}}\left[\bigoplus_{[{\bf t}]\in {\bf P}(\alpha)/G} H^{i-2\iota\left(\eta,Z^{[{\bf t}]}_{\eta}\right)} \left(Z^{[{\bf t}]}_{\eta}\,,\,\mathbb{C}\right)\right]\,.
		\end{align} 
	\end{definition}
	The vector space $H^*_{CR}\left(\Malpha/\Gamma\right) := \bigoplus_{i} H^i_{CR}\left(\Malpha/\Gamma\right)$ also has a multiplicative structure. To define it, first consider a differential form $\omega$ on $Z^{[{\bf t}]}_{\eta}$ (\ref{component description}), and let $\widetilde{\omega}$ be its pull-back on $f(\mathcal{N}^{\bf t}_{\eta})$, which is a $\Gamma$--invariant differential form. 
	
	The orbifold integration of $\omega$ is defined as
	\begin{align}
		\int_{Z^{[{\bf t}]}_{\eta}}^{orb}\omega := \dfrac{1}{|\Gamma|}\int_{f(\mathcal{N}^{\bf t}_{\eta})} \widetilde{\omega} \,.
	\end{align}
	Below, we follow the convention of taking $H^i(Y,\mathbb{C}) =0$ whenever $i\in \mathbb{Q}$ is not an integer.
	\begin{definition}\label{chen-ruan pairing}
		Let $d= \dim(\Malpha/\Gamma)$\,. For each rational number $0\leq n\leq 2d$, the  Chen--Ruan Poincar\'e pairing
		\[\langle,\rangle_{CR}\,:\, H^n_{CR}(\Malpha/\Gamma\,,\,\mathbb{C})\times H^{2d-n}_{CR}(\Malpha/\Gamma\,,\,\mathbb{C})\longrightarrow\mathbb{C}\]
		is a non-degenerate bilinear pairing. It is defined by combining the bilinear maps
		\[\langle,\rangle_{CR,[{\bf t}]}^{(\eta,\tau)}\,:\,H^{n-2\iota\left(\eta,Z^{[{\bf t}]}_{\eta}\right)}\left(Z^{[{\bf t}]}_{\eta},\mathbb{C}\right)\, \times\, H^{2d-n-2\iota\left(\tau,Z^{[{\bf t}]}_{\tau}\right)}\left(Z^{[{\bf t}]}_{\tau},\mathbb{C}\right)\longrightarrow \mathbb{C}\]
		as follows: 
		\begin{enumerate}
			\item if $\eta = \tau \simeq \mathcal{O}_X$, the Chen--Ruan pairing is the usual Poincar\'e pairing for the singular cohomology of $\Malpha/\Gamma$.
			\item if $\tau = \eta^{-1}$ then
			\[\langle \omega,\omega'\rangle^{(\eta,\eta^{-1})}_{CR,[{\bf t}]} := \int_{Z^{[{\bf t}]}_{\eta}}^{orb}\omega\wedge \omega'\]
			where $\wedge$ is the ordinary cup product on the singular cohomology of $Z^{[{\bf t}]}_{\eta}=f(\mathcal{N}^{s([{\bf t}])}_{\tau})/\Gamma$.
			\item if $\tau\neq\eta^{-1}$ then  
			\[\langle \omega,\omega'\rangle^{(\eta,\tau)}_{CR,[{\bf t}]} := 0\,.\]
		\end{enumerate}
	\end{definition}
	The Chen--Ruan product (denoted by $\cup$) is defined using the above pairing as follows. Let $\omega_i\in H^{k_i}\left((\Malpha)^{\eta_i}\,,\,\mathbb{C}\right)$ for $1\leq i \leq 2$, and set $\eta_3 = ( \eta_1\otimes\eta_2)^{-1}$. The Chen--Ruan product 
	$$\omega_1\cup \omega_2\in H^{k_1+k_2}\left((\Malpha)^{\eta_3}\,,\,\mathbb{C}\right)$$ 
	is defined by the requirement that it should satisfy
	\[\langle \omega_1\cup \omega_2\,,\,\omega_3\rangle_{CR} := \int_{\mathcal{S}/\Gamma}^{orb} e_1^*\omega_1\wedge e_2^*\omega_2\wedge e_3^*\omega_3\wedge c_{top}(\mathcal{F}_{\eta_1,\eta_2})\]
	for all $\omega_3\in H^*\left((\Malpha)^{\eta_3}\,,\,\mathbb{C}\right)$, where 
	\[\mathcal{S} := (\Malpha)^{ \eta_1}\cap (\Malpha)^{\eta_2}\,,\]
	and $e_i: \mathcal{S}/\Gamma \hookrightarrow (\Malpha)^{\eta_i}/\Gamma$ are the inclusion maps. Here $c_{top}(\mathcal{F}_{\eta_1,\eta_2})$ is the top Chern class of the orbifold obstruction bundle $\mathcal{F}_{\eta_1,\eta_2}$ on $\mathcal{S}/\Gamma$\,.
	
	We now provide a partial description of the Chen--Ruan products in certain cases. Clearly, the product will be zero if the intersection $\mathcal{S}$ becomes empty. The lemma below provides a necessary condition for it to be non-empty when the rank $r$ is prime.
	\begin{lemma}\label{chen-ruan product lemma}
		Let $\eta,\tau$ be two non-trivial elements in $\Gamma$\, of same order. We have $$\mathcal{S}= (\Malpha)^{\eta}\cap(\Malpha)^{\tau}\neq \emptyset\,\,\,\,\textnormal{only when}\,\,\, \langle \eta\rangle = \langle \tau\rangle\,.$$
		In particular, if $r$ is prime, we have $\mathcal{S}=\Metaalpha\cap\,(\Malpha)^{\tau}\neq \emptyset$\,\, only when\,\, $\eta \simeq \tau^k$ for some $k$.
	\end{lemma}
	\begin{proof}
		The idea is adapted from \cite[\S4.1]{Na05}. Let $$\gamma_{\eta}: Y_{\eta}\longrightarrow X\,\,\textnormal{and}\,\,\gamma_{\tau}: Y_\tau \longrightarrow X$$
		be the corresponding spectral curves. From Lemma \ref{parabolic fixed point} we know that there exist a space $\mathcal{N}_{\eta}$ with the property that every $E_*\in (\Malpha)^{\eta}$ is of the form $\sigma^*(F_*)$ for some $F_*\in \mathcal{N}_{\eta}$\,, and moreover, by Corollary \ref{quotient},
		$$(\Malpha)^{\eta}= \mathcal{N}_{\eta}/{\rm Gal}(\gamma_{\eta})\,.$$
		
		From the commutativity of the diagram
		
		\begin{align}
			\xymatrix@=2cm{\mathcal{N}_{ \eta} \ar[r]^{\_\otimes\gamma_{ \eta}^*( \tau)} \ar[d]_{\gamma_{ \eta*}} & \mathcal{N}_{ \eta} \ar[d]^{\gamma_{ \eta*}} \\
				(\Malpha)^{ \eta} \ar[r]^{\_\otimes  \tau} & (\Malpha)^{ \eta}
			}
		\end{align}
		it immediately follows that
		\begin{align}\label{intersection description}
			(\Malpha)^{ \eta}\cap (\Malpha)^{ \tau} = \gamma_{ \eta*}\left(\underset{\sigma\in {\rm Gal}(\gamma_{ \eta})}{\bigcup}\{F_*\in \mathcal{N}_{ \eta}\,\mid\,F_*\otimes\gamma_{ \eta}^*( \tau)\simeq \sigma^*(F_*)\}\right)\,.
		\end{align}
		Let $m=ord(\eta)=ord(\tau)$\,. As the isomorphism of parabolic bundles $F_*\otimes\gamma_{ \eta}^*( \tau)\simeq \sigma^*(F_*)$ in (\ref{intersection description}) is possible only when $\sigma = 1$ (see the proof of Lemma \ref{parabolic fixed point}), we get
		\begin{align}
			&F_*\otimes \gamma_{ \eta}^*( \tau) \simeq F_*\\
			&\implies \gamma_{\eta}^*( \tau^{r/m}) \simeq \mathcal{O}_{Y_{ \eta}}\,\,\,\,\,\,(\textnormal{taking determinant})\\
			&\implies \tau^{r/m}\in \langle\eta\rangle \,\,\,\,\,\,\,\,\,\,\,\,\,\,\,\,\,\text{\cite[Proposition 5.1]{GO18}}
		\end{align} 
		Of course, $\tau^{m}\simeq \mathcal{O}_X\in \langle\eta\rangle$ as well. Since $r$ is a product of distinct primes by assumption, clearly $\textnormal{gcd}(m,r/m)=1$\,; now it easily follows that
		$$\tau\in\langle\eta\rangle\,.$$ 	
		Of course, this argument is also valid if we interchange $\eta$ and $\tau$, in which case we get $\eta\in\langle\tau\rangle$. The result now follows. 
	\end{proof}
	From this, and the non-degeneracy of the Chen--Ruan product, we can immediately conclude the following. 
	\begin{corollary}
		Let $\omega_i\in H^*\left((\Malpha)^{\eta_i}\,,\,\mathbb{C}\right)$ for $1\leq i \leq 2$ with $\eta_i$'s non-trivial. Then $\omega_1\cup\omega_2 = 0$ whenever  $ord(\eta_1)= ord(\eta_2)$ and $\langle\eta_1\rangle\neq\langle\eta_2\rangle$\,.\\
		In particular, if $r$ is prime, we have Then $\omega_1\cup\omega_2 = 0$ whenever $\eta_1\notin \langle \eta_2\rangle$\,.
	\end{corollary}

	\section*{Acknowledgements}
	
	The authors would like to thank the anonymous referee for his/her careful reading of the manuscript and for many helpful comments and suggestions, which have helped to improve the manuscript to a large extent. The first-named author is partially supported by
	a J. C. Bose Fellowship (JBR/2023/000003).


\begin{thebibliography}{AAAA} 
		
		
		
		\bibitem[ALR]{ALR07} A. Adem, J. Leida and Y. Ruan, \textit{Orbifolds and Stringy Topology}, Cambridge University Press, {\bf 171} (2007).
		
		\bibitem[BDS]{BDS22} I. Biswas, P. Das and A. Singh, Chen--Ruan cohomology and moduli spaces of parabolic bundles over a Riemann surface, \textit{Proc. Math. Sci.} {\bf 132}, 31 (2022).
		
		\bibitem[BD]{BD10} I. Biswas and A. Dey, Chen–Ruan cohomology of some moduli spaces of parabolic vector bundles, \textit{Bull. Sci. math.} {\bf 134} (2010), 54--63.
		
		
		\bibitem[BH]{BH10} I. Biswas and A. Hogadi, Brauer group of moduli spaces of $\text{PGL}(r)$-bundles over a curve, \textit{Adv. Math.} {\bf 225} (2010), 2317--2331.
		
		\bibitem[BM]{BM19} I. Biswas and F. Machu, On the direct images of parabolic vector 
		bundles and parabolic connections, \textit{Jour. Geom. Phys.} {\bf 135} (2019), 219--234.
		
		\bibitem[BP1]{BP08} I. Biswas and M. Poddar, Chen–Ruan cohomology of some moduli spaces, \textit{Int. Math. Res. Not.} (2008), pp. 32.
		
		\bibitem[BP2]{BP10} I. Biswas and M. Poddar, Chen--Ruan cohomology of some moduli spaces II, \textit{Internat. J. Math.} {\bf 21(4)} (2010), 497--522.
		
		\bibitem[BR]{BR96} I. Biswas and N. Raghavendra, Canonical generators of the cohomology of moduli of parabolic bundles on curves, \textit{Math. Ann.} \textbf{306} (1996), 1--14.
		
		\bibitem[BY]{BY99} H. Boden and K. Yokogawa, Rationality of the moduli space of Parabolic bundles, \textit{Jour. London Math. Soc.} {\bf 59} (1999), 461--478.
		
		\bibitem[CR1]{CR02} W. Chen and Y. Ruan, Orbifold Gromov–Witten theory, in \textit{Orbifolds in Mathematics and
			Physics,} \textit{Contemp. Math.} {\bf 310} (2002), 25--85.
		
		\bibitem[CR2]{CR04} W. Chen and Y. Ruan, A new cohomology theory of orbifold. \textit{Comm. Math. Phys.}, {\bf 248} (2004), 1--31.
		
		\bibitem[DHVW]{DHVW85} L. Dixon, J. Harvey, C. Vafa and E. Witten, Strings on orbifolds I, \textit{Nucl. Phys. B} {\bf 261} (1985), 678--686.
		
		\bibitem[EV]{EV92} H. Esnault and E. Viehweg, {\it Lecturs on vanishing theorems}, Springer Basel AG, 1992.
		
		\bibitem[FG]{FG03} B. Fantechi and L. G\"ottsche, Orbifold cohomology for global quotients, \textit{Duke Math. J.} {\bf 117(2)} (2003), 197--227.
		
		\bibitem[GO]{GO18} P. B. Gothen and A. Oliveira, Topological Mirror symmetry for	parabolic Higgs bundles, \textit{Jour. Geom. and Phy.} {\bf 137} (2019), 7--34. 
		
		\bibitem[H]{H00} Y. I. Holla, Poincar\'e polynomial of the moduli spaces of parabolic bundles, \textit{Proc. Indian Acad. Sci. (Math. Sci.)} {\bf 110 (3)} (2000), 233--261.
		
		\bibitem[HH]{HH90} F. Hirzebruch and T. H\"ofer, On the Euler number of an orbifoid, \textit{Math. Ann.} {\bf 286} (1990), 255--260
		
		\bibitem[HP]{HP12} T. Hausel and C. Pauly, Prym varieties of spectral covers, \textit{Geom. Topol.} {\bf 16} (2012), 1609--1638.
		
		\bibitem[MS]{MS80} V. B. Mehta and C. S. Seshadri, Moduli of vector bundles on curves with parabolic structure. \textit{Math. Ann.} {\bf 248} (1980), 205--239.	
		
		
		\bibitem[Na]{Na05} F. Nasser, Torsion Subgroups of Jacobians Acting on Moduli Spaces of Vector Bundles. Ph.D Thesis, University of Aarhus, 2005.
		
		\bibitem[Ni]{Ni86} N. Nitsure, Cohomology of moduli of parabolic vector bundles, \textit{Proc. Indian Acad. Sci. Math. Sci.} \textbf{95} (1986), 61--77.
		
		\bibitem[NR]{NR75} M. S. Narasimhan and S. Ramanan, Generalized Prym varieties as fixed points. \textit{J. Indian Math. Soc.} {\bf 39}	(1975), 1--19.
		
		\bibitem[R]{R02} Y. Ruan, Stringy geometry and topology of orbifolds. In \textit{Symposium in Honor of C. H.
			Clemens}, 187--233. Contemporary Mathematics 312. Providence, RI: American Mathematical
		Society (2002).
		
		\bibitem[Se]{Ses} C. S. Seshadri, {\it Fibr\'es vectoriels sur les courbes alg\'briques}, Notes written by J.-M. Drezet from a course at the \'Ecole Normale Sup\'erieure, June 1980, 
		Ast\'erisque, 96. Soci\'et\'e Math\'ematique de France, Paris, 1982.
		{\it Ast\'erisque}, no. 96 (1982).
		
		\bibitem[St]{stacks-project} Stacks Project.
		
		\bibitem[Y]{Y04} T. Yasuda, Twisted jets, motivic measures and orbifold cohomology. \textit{Compos. Math.} {\bf 140} (2004), 396--422.
		
	\end{thebibliography}
\end{document}